\documentclass[landscape,onecolumn,11pt,3p]{elsarticle}

\usepackage{balance}
\usepackage{amsmath,graphicx}
\usepackage{amsfonts}
\usepackage{color}
\usepackage{amssymb}
\usepackage{amsthm}
\usepackage{bm}
\usepackage{cite}
\newtheorem{lemma}{Lemma}

\newtheorem{corollary}{Corollary}
\usepackage{cuted}
\usepackage{array,multirow}
\usepackage{stfloats}
\usepackage{bbm}
\usepackage{mathtools}
\usepackage{algorithm,algorithmic} %

\newcommand{\ist}{\hspace*{.3mm}}
\newcommand{\Newrevis}{\textcolor{black}}

\journal{Information Fusion}

\begin{document}

\begin{frontmatter}


\title{\LARGE Arithmetic Average Density Fusion - Part \cal{I}: Some Statistic and Information-theoretic Results} 

 \author[label1]{\large Tiancheng Li}  
 \author[label1]{\large Yan Song}  
 \author[label1]{\large Enbin Song}  
 \author[label1]{\large Hongqi Fan}  
 \address[label1]{Key Laboratory of Information Fusion Technology (Ministry of Education), School of Automation, Northwestern Polytechnical University, Xi'an 710072, China.
 E-mail: t.c.li@nwpu.edu.cn,syzx@mail.nwpu.edu.cn} 
\address[label2]{College of Mathematics, Sichuan University, Chengdu, China, e-mail: e.b.song@163.com}
\address[label3]{Key Laboratory of Science and Technology on ATR, National University of Defense Technology, Changsha 410073, Hunan, China, e-mail: fanhongqi@nudt.edu.cn}

\begin{abstract}
Finite mixture such as the Gaussian mixture is a flexible and powerful probabilistic modeling tool for representing the multimodal distribution widely involved in many estimation and learning problems. The core of it is representing the target distribution by the \textit{arithmetic average} (AA) of a finite number of sub-distributions which constitute a mixture. While the mixture has been widely used for single sensor {filter} design, it is only recent that the AA fusion demonstrates compelling performance for multi-sensor {filter} design.
In this paper, some statistic and information-theoretic results are given on \Newrevis{the covariance consistency, mean square error, mode-preservation capacity, and the information divergence of the AA fusion approach}. In particular, based on the concept of \textit{conservative fusion}, the relationship of the AA fusion with the existing conservative fusion approaches such as \textit{covariance union} and \textit{covariance intersection} 
is exposed. A suboptimal weighting approach has been proposed, which {jointly with the best mixture-fit property of the AA fusion} leads to a max-min optimization problem. 
Linear Gaussian models are considered for algorithm illustration and simulation comparison, resulting in the first-ever AA fusion-based multi-sensor Kalman filter. 
\end{abstract}


\end{frontmatter}

\section{Introduction}



T{he} last two decades have witnessed a steady uptick in the application
of the linear information fusion approach such as the averaging operation to the
{multi-sensor} estimation problem \citep{Sayed14book,Li22chapter,Li23BM}. 
Two 
common types of estimator information that need to be averaged are variables and density functions. 
They 
correspond to the two classes of estimators: the point estimator such as the most known Kalman filter of which the estimate is given in terms of a point estimate associated with covariance and the density estimator such as {the particle filter of which the estimate is given in terms of a density function that approximates the Bayesian posterior}. Both types of fusion {have} seen substantial interest for multi-agent collaboration thanks to the vitalization of networked sensors/systems \citep{Sayed14book,Javadi20RadarNet,Li22chapter,Li23BM}. 
Fusing densities is superior in the sense that a full description about the unknown variable of interest is utilized
rather than some statistics of it. It is
also compatible with heterogeneous data models across the agents \citep{Li17PC,Li17PCsmc,Kayaalp22}. 
While the averaging of variables results in a variable of nice statistical property \citep{Li19Second}, 
the averaging of distributions leads to a finite mixture distribution (FMD). 
In the FMD, mixands are properly weighted to 
jointly approximate the target probability distribution $p(\mathbf{x})$ by their arithmetic average (AA):
\begin{equation}\label{eq:defAA}
  f_\text{AA}(\mathbf{x}) \triangleq \sum_{i \in \mathcal{I}} {w_i}f_i(\mathbf{x}) \ist,
\end{equation}
where $\mathbf{x}$ denotes the state of interest, $\mathbf{w} \triangleq [w_1, w_2, ...,w_I]^\mathrm{T}$ are non-negative, normalized mixing/fusing weights, and $f_i(\mathbf{x})$ are {some distribution functions} (of the same family or not), e.g., the
probability density function (PDF) and probability hypothesis density function \citep{Li22RFS-AA-Derivation} 
yielded by a finite set of estimators $i\in \mathcal{I}=\{1,2,\cdots,I\}$ conditioned on different observations, models or hypotheses. 

%

{The need for FMD{, which can be dated back to \citep{Pearson1894},} 
commonly arises in dealing with the uncertainties due to
multi-modal data/noise \citep{Burnham02}, 
stochastically switched models of the real state \citep{Sylvia06}, 
and multiple objects \citep{Vo15mtt}.} 
{It} facilitates the recursive Markov-Bayesian filtering calculation greatly in two means: First, a mixture of conjugate priors is also conjugate and can approximate any kind of prior \citep{Dalal83mixturePrior,Diaconis83mixturePrior}. Second, the linear fusion of a finite number of mixtures of the same parametric family remains a mixture of that family. These properties play a key role in the mixture filters such as the {celebrated}
Gaussian mixture (GM) filter, 
and multi-Bernoulli mixture filters \citep{Vo13Label,Williams15taesPMBM,Angel18PMBMdeivation}. 
In fact, {as one of the most known opinion pools since \citep{Stone61LOP}}, FDM 
has also been widely used in machine learning \citep{McLachlan00} {such as data clustering}. 
{In spite of these,} 
it is only recent that the linear AA fusion has been intensively used for
multi-sensor {distribution fusion based on random set filters (see the review given in \citep{Li22chapter} and \citep{Li22RFS-AA-Derivation})} 
which has been proven highly efficient in computation, robust/tolerant to
individual sensor misdetection, and insensitive to internode correlation \citep{Bailey12}. 

{In the context of multi-sensor density fusion, each sensor computes a density \Newrevis{such as the Bayesian posterior} based on its own observations and shares it with its inter-connected neighbors 
so that these densities can be fused with weights in an appreciate means for more accurate and robust estimate. In particular, the AA density fusion is well suited for distributed implementation, which is advantageous as compared with the centralized fusion.}
\Newrevis{Recently, the rationale and nice properties of the AA fusion as well as some other density fusion approaches have been further analyzed in \citep{Koliander22} and extension to fusion of {distributions over unknown quantities of interest} and of soft-decisions has been made in \citep{Kayaalp22} and \citep{Kayaalp23}, respectively. Yet, it is so far unknown what happens when this fundamental fusion method meets the benchmark Kalman/partical filter.}
Nevertheless, theoretical study on the AA density fusion approach 
is still short in two aspects, which motivate our work.
\begin{enumerate}
  \item First, while the concept of conservative fusion \citep{Uhlmann95,Julier01,Uhlmann03,Julier05cu,Reece05,Julier06,Bochardt06,Ajgl14,Lubold19formal,Noack17ICI,Wang21cu} 
      has been well accepted, the AA fusion 
gains the mentioned advantages at the price of an inflated covariance which seems at variant with the minimum variance estimator/fusion. This unavoidably raises a concern on the accuracy of the estimator if one simply swaps the inflated variance with increased mean square error (MSE). 
  We in this paper clarify their fundamental difference and provide a comprehensive analysis of the statistics of the AA fusion and its connection with existing conservative fusion approaches {including a direct comparison of their fused covariances. \Newrevis{It is for the first time shown how the AA density fusion reduces the MSE and demonstrates consistency under certain conditions}. Moreover, we point out an important property of the AA fusion in preserving the modes of the fusing estimators.} 
  \item Second, while the AA fusion has been well tailored for various multiple-object filter fusion (see the review give in \citep{Li22chapter} and \citep{Li22RFS-AA-Derivation}), it remains unclear how the mixture as a whole compares with the mixands (especially the \textit{best} one) and how the fusing weights should be designed proactively in order to maximize the fusion gain {in the context of multi-sensor filter fusion}. 
We provide information-theoretic results to answer these questions, \Newrevis{demonstrating how} the AA fits the target distribution better than the average of the fusing estimators and even than the best one.
\end{enumerate}

These statistic and information-theoretic findings are expected to underpin the use of the FMD and AA fusion approach. Empirical analysis and simulation study based on the most known Gaussian distributions and the benchmark Kalman and particle filters are provided to demonstrate the theoretical results. \Newrevis{Based on these theoretical results, we propose} the first-ever AA fusion-based multi-sensor Kalman and particle filters which are computationally faster and yield the comparable estimate accuracy as the state of the art.
Some results have been directly used for Student's $t$ density fusion in the presence of outlier \citep{Li22StudentAA} while a unified derivation of the AA fusion for multi-sensor random finite set filters is provided and realized in \Newrevis{the companion papers
\citep{Li22RFS-AA-Derivation,Li23Heterogeneous,Li23HeterVA}. This paper serves as the first part of a series of papers that provide a comprehensive and  thorough study of the AA fusion methodology and its application for target tracking.}

The remainder of this paper is organized as follows. Representative conservative fusion approaches are analyzed and compared in section \ref{sec:statistics}, highlighting the 
\Newrevis{statistic properties}
of the AA fusion approach. 
The divergence of the mixture from the target distribution is studied in section \ref{sec:information-theoretics}, providing an information-theoretic criterion for the AA fusion. Furthermore, we propose a suboptimal approach to fusing weight design. Comprehensive comparison of these conservative fusion approaches in the context of target tracking is presented in section \ref{sec:simulation} before the paper is concluded in section \ref{sec:conclusion}. 



\section{Conservative Fusion and Statistics of AA Fusion} \label{sec:statistics}

In the context of time-series estimation, optimality such as minimum MSE (MMSE) and minimized Bayes risk \citep{Kay93} is desirable in different classes of optimal estimators: MMSE point estimator and Bayes-optimal density estimator. There is a key difference between two optimal criterion: the former relies on the statistics of the estimator such as the mean and variance \citep{Li03OptimalLinear}, 
and the latter on information measure of the posterior for which a proper distribution-oriented metric such as the Kullback Leibler (KL) divergence is useful. 
Note that 
{standard} statistics such as the mean and variance 
do not apply to the multi-target estimator. In this section, we limit analysis with respect to the single target only.

\Newrevis{In what follows, we first give the main notations, concepts and definitions used in this paper. Based on this, we compare the AA fusion method with the covariance union (CU) method and further with the covariance intersection (CI) method and its extension. Then, we analyze the MSE of the AA density fusion, demonstrating its consistency under certain conditions. Finally, we point out and illustrate the mode-preservation capacity of the AA fusion in preserving the modes of the fusing estimators.} 

\subsection{Notations, Concepts and Definitions}
In the following, we use $\mathbf{x} \in \mathbb{R}^{d}$ to denote the $d$-dimensional state of the target which is 
a random quantity to be estimated, namely the real/true state.
\Newrevis{Based on the Bayesian viewpoint,} we use $p(\mathbf{x})$ to denote the corresponding PDF, namely the probability distribution of $\mathbf{x}$. For a given Bayesian posterior $f(\mathbf{x})$ 
{conditioned on real-time observations}\footnote{Hereafter, we ignore the dependence on the random observations in the notation for simplicity unless when the random nature of the observations is addressed.
{This makes sense because given the observation at any particular time \Newrevis{for fusion}, what is fused is the specific posterior/density obtained at that time.} 
}, from which the 
state estimate can be attained in either of two common ways, namely the expected a posteriori (EAP) and maximum a posteriori (MAP) estimators as follows
\begin{align}
\hat{\mathbf{x}}^\mathrm{EAP} &= \int_{\mathbb{R}^{d}} \mathbf{\tilde{x}} f(\mathbf{\tilde{x}}) d \mathbf{\tilde{x}} \ist,  \label{eq:eap} \\
\hat{\mathbf{x}}^\mathrm{MAP} &= \mathop{\arg\sup}\limits_{\mathbf{\tilde{x}}\in \mathbb{R}^{d}} f(\mathbf{\tilde{x}}) \ist. \label{eq:map}
\end{align}
That is, in the EAP estimator, the state estimate is given as the mean of the posterior while in the MAP the state estimate is given as the peak/mode of the posterior distribution. 

We consider a number of estimate pairs, each composed of a state estimate $\hat{\mathbf{x}}_i$ and an associated positive-definite error covariance matrix $\mathbf{P}_i$, $ i \in \mathcal{I}$, which are to be fused using weights $\mathbf{w} \triangleq \{w_1,w_2,...,w_I\}$, where $w_i \geq 0, \mathbf{w}^\mathrm{T}\mathbf{1}_I = 1, \forall i \in \mathcal{I}$. Here $\mathbf{1}_I $ is the all-ones column vector of dimension $I$. 
Hereafter, $\hat{\mathbf{x}}_i$ is given by the EAP estimator unless otherwise stated. Then, each estimate pair corresponds to the first and second moments of the posterior PDF $f_i(\mathbf{x})$ which is an estimate of the real distribution $p(\mathbf{x})$, i.e., $\hat{\mathbf{x}}_i = \int_{\mathbb{R}^{d}} \mathbf{\tilde{x}} f_i(\mathbf{\tilde{x}}) d \mathbf{\tilde{x}}$, $\mathbf{P}_i= \int_{\mathbb{R}^{d}} (\mathbf{\tilde{x}}-\hat{\mathbf{x}}_i )(\cdot)^\mathrm{T} f_i(\mathbf{\tilde{x}}) d \mathbf{\tilde{x}}$ \footnote{Note that a Gaussian PDF can be uniquely determined by an estimate pair. Fusion in terms of only the mean and variance implicitly imposes Gaussian assumption. These being said, the AA density fusion is by no means limited to any specific distributions or finite moments of the distribution.}. Hereafter, we use the shorthand writing $(\mathbf{x}-\mathbf{y})(\cdot)^\mathrm{T} \triangleq (\mathbf{x}-\mathbf{y})(\mathbf{x}-\mathbf{y})^\mathrm{T}$.
The MSE of $\hat{\mathbf{x}}_i$ (conditional on the given observations of sensor $i$) is defined as \footnote{{In the Bayesian viewpoint, both the real state and the observations are random and so the complete expression of the MSE of $\hat{\mathbf{x}}_i \in \mathbb{R}^{d}$ given observation $\mathbf{y} \in \mathbb{R}^{d_y}$ can be given by
\begin{equation}\nonumber
  \Newrevis{\mathrm{CompleteMSE}}_{\hat{\mathbf{x}}_i} \triangleq \int_{\mathbb{R}^{d_y}} \int_{\mathbb{R}^{d}} (\mathbf{{x}}-\hat{\mathbf{x}}_i)(\cdot)^\mathrm{T}p(\mathbf{{x}},\mathbf{{y}})d\mathbf{{x}}d\mathbf{{y}} .
\end{equation}
where $p(\mathbf{{x}},\mathbf{{y}})$ is the joint distribution of the real state $\mathbf{x}$ and the observation $\mathbf{y}$ and $\hat{\mathbf{x}}_i$ is conditional on $\mathbf{y}$.}}
\begin{equation}\label{eq:MSE-def}
  \mathrm{MSE}_{\hat{\mathbf{x}}_i} \triangleq \mathrm{E}_p [(\mathbf{x}-\hat{\mathbf{x}}_i)(\cdot)^\mathrm{T}] = \int_{\mathbb{R}^{d}} (\mathbf{{x}}-\hat{\mathbf{x}}_i)(\cdot)^\mathrm{T}p(\mathbf{{x}})d\mathbf{{x}} \ist.
\end{equation}


\textit{Definition 1 (Conservative)}. An estimate pair ($\hat{\mathbf{x}},\mathbf{P}_{\hat{\mathbf{x}}}$) regarding the real state $\mathbf{x}$, is deemed conservative  \citep{Uhlmann95,Uhlmann03,Julier06,Bochardt06} when
\begin{equation}
\mathbf{P}_{\hat{\mathbf{x}}} \succeq \mathrm{MSE}_{\hat{\mathbf{x}}} \ist.
\end{equation}
That is, $\mathbf{P}_{\hat{\mathbf{x}}} - \mathrm{E}_p [(\mathbf{x}-\hat{\mathbf{x}})(\cdot)^\mathrm{T}]$ is positive (semi-)definite.

The notion is also referred to as \textit{covariance consistent} and as pessimistic definite \citep{XLi06}. Extended definition of the conservativeness of PDFs can be found in \citep{Ajgl14,Lubold19formal}.

%

With respect to the type of data, there are two forms of AA fusion as follows.

\textit{Definition 2 (AA $v$-fusion)}. In a point estimation problem, the AA $v$-fusion is carried out with regard to these state estimate variables $\hat{\mathbf{x}}_i$, $ i \in \mathcal{I}$ which yields a new variable 
\begin{equation}\label{eq:AA-v-fusion}
  \hat{\mathbf{x}}_\mathrm{AA} = \sum_{i \in \mathcal{I}} w_i\hat{\mathbf{x}}_i \ist.
\end{equation}

\textit{Definition 3 (AA $f$-fusion)}. In the Bayesian formulation \Newrevis{as the concern of this paper}, the estimation problem is to find a distribution that best fits $p(\mathbf{x})$. The corresponding AA $f$-fusion is carried out with regard to $f_i(\mathbf{x})$, $ i \in \mathcal{I}$ which yields a mixture of these fusing distributions, a FMD 
\begin{equation}\label{eq:AA-f-fusion}
  f_\mathrm{AA}(\mathbf{x}) = \sum_{i \in \mathcal{I}} w_if_i(\mathbf{x}) \ist.
\end{equation}

{\textit{Definition 4 (Naive fusion)}. For a set of independent estimate pairs $(\hat{\mathbf{x}}_i,\mathbf{P}_i)$, $i \in \mathcal{I}$, the naive fusion (NF) which is also referred to as the convex combination \citep{Chong00T-fusion} is given by
\begin{align}
\hat{\mathbf{x}}_\mathrm{naive} &=  \mathbf{P}_\mathrm{naive} \sum_{i \in \mathcal{I}} \mathbf{P}_i^{-1} \hat{\mathbf{x}}_i \ist, \label{eq:naive-x}\\
\mathbf{P}_\mathrm{Naive} & = \Big(\sum_{i \in \mathcal{I}}\mathbf{P}_i^{-1}\Big)^{-1} \ist,\label{eq:naive-P}
\end{align}
which corresponds to the product of Gaussian PDFs with mean $\hat{\mathbf{x}}_i$ and covariance $\mathbf{P}_i, i \in \mathcal{I}$, respectively, i.e., $\mathcal{N}(\mathbf{x}; \hat{\mathbf{x}}_\mathrm{naive}, \mathbf{P}_\mathrm{Naive}) = \prod_{i \in \mathcal{I}}\mathcal{N}_i(\mathbf{x};\hat{\mathbf{x}}_i,\mathbf{P}_i)$.} 

It is often too ideal to assume independence among the fusing estimators especially when an inter-connect sensor network is considered. 
Then, multi-sensor optimal fusion in the sense of whether MMSE or Bayes
needs to quantify exactly the cross-correlation among the sensors \citep{Chong90,Li03OptimalLinear,Sun20OLF}. 
Unfortunately, this often turns out to be impractical/inconcervable due to the complicated, latent correlation among sensors/agents and so one may resort to suboptimal, correlation-insensitive solutions \citep{Uhlmann95,Julier01,Uhlmann03,Julier05cu,Reece05,Julier06,Bochardt06,Noack17ICI,Wang21cu} including the AA fusion. 

\subsection{\Newrevis{Comparison between CU and AA}}
%
Based on the concept of conservativeness, there are a number of results as given in the following Lemmas.
 \begin{lemma} \label{lemma_PCU}
 For a set of estimate pairs $(\hat{\mathbf{x}}_i,\mathbf{P}_i)$, $i \in \mathcal{I}$ in which at least one is conservative, a sufficient condition for the fused estimate pair ($\hat{\mathbf{x}}_\mathrm{AA},\mathbf{P}_\mathrm{CU}$) to be conservative is that
 \begin{equation} \label{eq:CU_P}
 \mathbf{P}_\mathrm{CU} \succeq \mathbf{P}_i + (\hat{\mathbf{x}}_\mathrm{AA}-\hat{\mathbf{x}}_i)(\cdot)^\mathrm{T}, \hspace{1mm} \forall i \in \mathcal{I} \ist,
 \end{equation}
for which a tight bound is given by
\footnote{We note that two matrixes may not be comparable for which we use $\mathbf{A} \succ \mathbf{B}$ if $\mathrm{Tr}(\mathbf{A}) > \mathrm{Tr}(\mathbf{B})$, where $\mathrm{Tr}(\mathbf{A})$ calculates the trace (or the determinant) of matrix $\mathbf{A}$. For two incomparable matrixes $\mathbf{A}, \mathbf{B}$ with the same trace (or the same determinant), one may further use $\mathbf{A} = \max(\mathbf{A}, \mathbf{B})$. }
 \begin{equation} \label{eq:CU_P-max}
 \mathbf{P}_\mathrm{CU} \succeq \mathbf{P}_\mathrm{CU}^{u}  \triangleq \max_{i \in \mathcal{I}} \big(\mathbf{P}_i + (\hat{\mathbf{x}}_\mathrm{AA}-\hat{\mathbf{x}}_i)(\cdot)^\mathrm{T}\big) \ist.
 \end{equation}
 \end{lemma}

Proof for {\eqref{eq:CU_P} can be found in \citep{Li17PC} where an implicit assumption is made as $\mathbf{P}_i$ is uncorrelated with $\hat{\mathbf{x}}_\mathrm{AA}$}. It actually provides a conservative fusion method which is known as covariance union (CU) \citep{Uhlmann03,Julier05cu,Wang21cu}. It is deemed \textit{fault tolerant} as it preserves covariance consistency as long as at least one fusing estimator is conservative. When all fusing estimators are conservative, i.e., $\mathbf{P}_i \succeq \mathrm{E}_p [ (\mathbf{x}-\hat{\mathbf{x}}_i)(\cdot)^\mathrm{T}], \hspace{1mm} \forall i \in \mathcal{I}$, one has the following result

\begin{corollary} \label{corollary_PCU-l}
For a set of conservative estimate pairs $(\hat{\mathbf{x}}_i,\mathbf{P}_i)$, $i \in \mathcal{I}$, a sufficient condition for the fused estimate pair ($\hat{\mathbf{x}}_\mathrm{AA},\mathbf{P}_\mathrm{CU}^{l}$) to be conservative is given by a lower bound as compared with \eqref{eq:CU_P-max}
 \begin{equation} \label{eq:CU_P-min}
 \mathbf{P}_\mathrm{CU} \succeq \mathbf{P}_\mathrm{CU}^{l}  \triangleq \min_{i \in \mathcal{I}} \big(\mathbf{P}_i + (\hat{\mathbf{x}}_\mathrm{AA}-\hat{\mathbf{x}}_i)(\cdot)^\mathrm{T}\big) \ist.
 \end{equation}
\end{corollary}

\begin{lemma} \label{lemma_AAmeanVar}
For Gaussian distributions $f_i(\mathbf{x}) = \mathcal{N}(\mathbf{x}; \hat{\mathbf{x}}_i,\mathbf{P}_i)$, $ i \in \mathcal{I} $, the AA $f$-fusion \eqref{eq:AA-f-fusion} results in a FMD for which the mean and covariance are respectively given by \eqref{eq:AA-v-fusion} and
\begin{align}
\mathbf{P}_\mathrm{AA} &= \sum_{i \in \mathcal{I}} w_i\tilde{\mathbf{P}}_i \ist, \label{eq:AA-f-P}
\end{align}
where the adjusted covariance matrix is given by 
\begin{equation} \label{eq:CovAdjut}
\tilde{\mathbf{P}}_i \triangleq \mathbf{P}_i + (\hat{\mathbf{x}}_\mathrm{AA}-\hat{\mathbf{x}}_i)(\cdot)^\mathrm{T} \ist.
\end{equation}

 \end{lemma}
 \begin{proof} First, $
\hat{\mathbf{x}}_\mathrm{AA} 
= \int_{\mathbb{R}^{d}} \mathbf{\tilde{x}} \sum_{i \in \mathcal{I}} w_if_i(\mathbf{\tilde{x}}) d \mathbf{\tilde{x}}
= \sum_{i \in \mathcal{I}} w_i\hat{\mathbf{x}}_i$.
Proof of \eqref{eq:AA-f-P} for fusing two Gaussian distributions can be found in Appendix B of \citep{Li20AAmb}, which {is straightforward to be} extended to any finite number of fusing distributions.
 \end{proof}
The Appendix B of \citep{Li20AAmb} further showed that the above results $\hat{\mathbf{x}}_\mathrm{AA}$ and $\mathbf{P}_\mathrm{AA}$ correspond to the first and second moments of the resulting Gaussian distribution by merging \citep{Salmond09} all mixands in the mixture. {We note that hereafter, the component merging means the method proposed by \citep{Salmond09} unless otherwise stated.} In fact, the Gaussian PDF that best fits the AA mixture has the same first and second moments \citep[Theorem 2]{Runnalls07}, i.e., 
\begin{equation}\label{eq:best-fit-GM-Gaussian}
  (\hat{\mathbf{x}}_\mathrm{AA}, \mathbf{P}_\mathrm{AA})= \mathop{\arg\min}\limits_{(\mathbf{\mu},\mathbf{P})}D_\text{KL}\big( f_\text{AA}\| \mathcal{N}(\mathbf{\mu},\mathbf{P})\big) \ist,
\end{equation}
where 
${D_{{\rm{KL}}}}(f||p) \triangleq \int_{\mathbb{R}^{d}} {f(\mathbf{x})\log \frac{f(\mathbf{x})}{p(\mathbf{x})}d \mathbf{x}}$ denotes the KL divergence of the probability distribution $p(\mathbf{x})$ relative to $f(\mathbf{x})$. 



\Newrevis{
\begin{lemma} Fusing the same group of estimate pairs, the AA fusion (using weights smaller than unit) is less conservative in comparison with the CU fusion in the following sense
\begin{equation}\label{eq:AAvsCU}
  \mathbf{P}_\mathrm{AA} \preceq \mathbf{P}_\mathrm{CU} \ist,
\end{equation}
where the equation holds if and only if (iff) $\tilde{\mathbf{P}}_i = \tilde{\mathbf{P}}_j, \forall i \neq j$.
\end{lemma}
\begin{proof}
The result is straightforward from \eqref{eq:CU_P} and \eqref{eq:AA-f-P}.
\end{proof}
}

\subsection{Some Other Conservative Fusion Approaches}

In contrast to the AA fusion \eqref{eq:AA-f-fusion}, the geometric average (GA) of the fusing sub-PDFs $f_i(\mathbf{x})$ is given as
\begin{equation}\label{eq:GA-f-fusion}
  f_\mathrm{GA}(\mathbf{x}) = C^{-1}\prod_{i \in \mathcal{I}}f_i^{w_i}(\mathbf{x}) \ist,
\end{equation}
where $C\triangleq \big(\int_{\mathbb{R}^{d}} \prod_{i \in \mathcal{I}}f_i^{w_i}(\mathbf{\tilde{x}}) d\mathbf{\tilde{x}}\big)^{-1}$ is the normalization constant.

Obviously, the GA fusion is a log-linear fusion, i.e., $\log f_\mathrm{GA}(\mathbf{x}) = \sum_{i \in \mathcal{I}} {w_i}\log f_i(\mathbf{x})$.
In the specific case of Gaussian estimate pairs $(\hat{\mathbf{x}}_i,\mathbf{P}_i), i \in \mathcal{I}$ , the GA is given as follows
\begin{align}
\hat{\mathbf{x}}_\mathrm{GA}(\mathbf{w}) &=  \mathbf{P}_\mathrm{GA} \sum_{i \in \mathcal{I}} w_i\mathbf{P}_i^{-1} \hat{\mathbf{x}}_i \ist, \label{eq:CI-x}\\
\mathbf{P}_\mathrm{GA}(\mathbf{w}) &= \Big(\sum_{i \in \mathcal{I}} w_i\mathbf{P}_i^{-1}\Big)^{-1} \ist. \label{eq:CI-P}
\end{align}

When the fusing weights are abandoned, \eqref{eq:CI-x} and \eqref{eq:CI-P} result in the NF \eqref{eq:naive-x} and \eqref{eq:naive-P}, respectively. 
As another special case of the {Gaussian-}GA fusion, the CI fusion \citep{Julier01,Julier06,Bochardt06} optimizes the fusing weights as follows
\begin{equation}\label{eq:CI-w}
  \mathbf{w}_\mathrm{CI} =  \mathop{\arg\min}\limits_{\mathbf{w} \in \mathbb{W}} \mathrm{Tr}(\mathbf{P}_\mathrm{GA}) \ist.
\end{equation}
where $\mathbb{W}\triangleq\{\mathbf{w} \in \mathbb{R}^{I}|\mathbf{w}^\mathrm{T}\mathbf{1}_I = 1, w_i \geq 0, \forall i \in \mathcal{I}\} \subset \mathbb{R}^{I}$.

That is, $\hat{\mathbf{x}}_\mathrm{CI} =\hat{\mathbf{x}}_\mathrm{GA}(\mathbf{w}_\mathrm{CI}) , \mathbf{P}_\mathrm{CI} = \mathbf{P}_\mathrm{GA}(\mathbf{w}_\mathrm{CI})$. In this line of research, a variety of approaches have been proposed for further reducing the error covariance metric, leading to various CI-like, less conservative fusion approaches such as the so-called split-CI \citep{Julier01} / bounded covariance inflation \citep{Reece05}, 
inverse CI (ICI) \citep{Noack17ICI}.
{For example, in contrast to \eqref{eq:CI-P} the resulted covariance of the ICI fusion is given as
\begin{equation}\label{eq:ici}
\mathbf{P}_\mathrm{ICI} = \Big(\sum_{i \in \mathcal{I}}\mathbf{P}_i^{-1} - \big(\sum_{i \in \mathcal{I}} w_i\mathbf{P}_i\big)^{-1}  \Big)^{-1} \ist.
\end{equation}}


In contrary, it is our observation 
that the GA fusion is often not too conservative but insufficient in cluttered scenarios which may suffer from out-of-sequential measurement \citep{Julier05cu}/spurious data \citep{Wang21cu}, {network attack} and model mismatch. The reason is simply that the fused covariance $\mathbf{P}_\mathrm{CI} $ as in \eqref{eq:CI-P} does not take into account $\hat{\mathbf{x}}_\mathrm{CI}$ or any fusing state estimate $\hat{\mathbf{x}}_i$ as both AA and CU fusion do by having $(\hat{\mathbf{x}}_\mathrm{AA}-\hat{\mathbf{x}}_i)(\cdot)^\mathrm{T}$. 
Then, a more conservative fusion approach like the AA and CU fusion becomes useful. Noticing this, a conservative fusion approach referred to FFCC (fast and fault-tolerant convex combination) that uses smaller fusing weights, namely, $\mathbf{w}^\mathrm{T}\mathbf{1}_I =\delta$, where $\delta \leq 1$, was presented in \citep{Wang09}. So, obviously, {$\mathbf{P}_\mathrm{FFCC} \succeq \mathbf{P}_\mathrm{CI}$} where the equation holds iff $\delta =1$. 

Summarizing the above results leads to a \textit{conservativeness chain} {as follows} 
\begin{corollary} \Newrevis{Fusing the same group of estimate pairs using the same weights (if needed), the resulted covariances of the naive fusion, ICI, CI, AA and CA fusion methods satisfy}
\begin{equation}\label{eq:conservationfusionchain}
\mathbf{P}_\mathrm{Naive} \prec \mathbf{P}_\mathrm{ICI} \prec  \mathbf{P}_\mathrm{CI} \preceq \mathbf{P}_\mathrm{AA} \preceq  \mathbf{P}_\mathrm{CU}  
\end{equation}
where $\mathbf{P}_\mathrm{CI} = \mathbf{P}_\mathrm{AA} = \mathbf{P}_\mathrm{CU}$ holds iff all fusing estimators are identical.
\end{corollary}

\subsection{\Newrevis{MSE of EAP estimator}} \label{sec:InflationAnalysis}

As addressed so far, the AA fusion yields a consistent fused covariance if all fusing estimators are covariance consistent. In this section, we analyze the fused mean in the case of EAP estimator, for which the AA $f$-fusion \eqref{eq:AA-f-fusion} will lead to the AA $v$-fusion \eqref{eq:AA-v-fusion} \citep{Li19Second}, i.e.,
\begin{align}\label{eq:AA-EAP}
  \hat{\mathbf{x}}_\mathrm{AA}^\mathrm{EAP} &= \int_{\mathbb{R}^{d}} \mathbf{\tilde{x}} f_\mathrm{AA}(\mathbf{\tilde{x}}) d \mathbf{\tilde{x}}  \nonumber \\
  &= \sum_{i \in \mathcal{I}} w_i \int_{\mathbb{R}^{d}} \mathbf{\tilde{x}}f_i(\mathbf{\tilde{x}}) d \mathbf{\tilde{x}}  \nonumber \\
  &= \sum_{i \in \mathcal{I}} w_i\hat{\mathbf{x}}_i^\mathrm{EAP} \ist.
\end{align}

Now, we consider the following \Newrevis{mean-cross-error (MCE) between the given fusing estimators 
with regard to the real state distribution $p(\mathbf{x})$, c.f., \eqref{eq:MSE-def}
\begin{align}\label{eq:state-trial}
  \mathrm{MCE}_p(\hat{\mathbf{x}}_i,\hat{\mathbf{x}}_j)
  &\triangleq \mathrm{E}_p [(\mathbf{x}-\hat{\mathbf{x}}_i)(\mathbf{x}-\hat{\mathbf{x}}_j)^\mathrm{T}] \nonumber \\
  &= \int_{\mathbb{R}^{d}} (\mathbf{{x}}-\hat{\mathbf{x}}_i)(\mathbf{x}-\hat{\mathbf{x}}_j)^\mathrm{T}p(\mathbf{{x}})d\mathbf{{x}} \ist.
\end{align}}
Further taking into account the random nature of the fusing estimators conditional on their respective observations \Newrevis{that are conditionally independent with each other, the unbiasedness} of the fusing estimators imply
\begin{align}
  \mathrm{MCE}_p(\hat{\mathbf{x}}_i,\hat{\mathbf{x}}_j) & =   \mathbf{0}, \forall i \neq j \ist. \label{eq:independent}
\end{align}


\begin{lemma} \label{lemma_AAmeanVar}
For a set of conditionally independent and unbiased estimators, the AA fusion \eqref{eq:AA-EAP} gains better accuracy in the sense that
\begin{equation} \label{eq:AAmseGain}
\sum_{i \in \mathcal{I}} w_i \big(\mathrm{MSE}_{\hat{\mathbf{x}}_i}- \mathrm{MSE}_{\hat{\mathbf{x}}_\mathrm{AA}}\big) \succ \mathbf{0} \ist.
\end{equation}
 \end{lemma}

\Newrevis{
\begin{proof}
The proof and experimental demonstration are given in Appendix \ref{Append-MSE}.
\end{proof}
}


{When $w_i = |\mathcal{I}|^{-1}, \forall i \in \mathcal{I}$, \eqref{eq:AA-mse-uncor} will reduce to
\begin{align}\label{mse-convergence}
  \mathrm{MSE}_{\hat{\mathbf{x}}_\mathrm{AA}} & = |\mathcal{I}|^{-1} \sum_{i \in \mathcal{I}} w_i \mathrm{MSE}_{\hat{\mathbf{x}}_i} \\
  & \leq |\mathcal{I}|^{-1}  \max_{i\in \mathcal{I}} \mathrm{MSE}_{\hat{\mathbf{x}}_i} \ist,
\end{align}
which indicates that the AA fusion can significantly benefit in gaining lower MSE (in comparison with the average MSE of all fusing estimators) in the case that all fusing estimators are conditionally independent of each other and unbiased. In fact, it is easy to see the MSE convergence of the fusion expressed by, $\forall \varepsilon >0$,
\begin{equation}\label{eq:MSEconvergence}
  \lim_{|\mathcal{I}|\rightarrow \infty} \text{Pr}[\mathrm{MSE}_{\hat{\mathbf{x}}_\mathrm{AA}} > \varepsilon ] =0 \ist.
\end{equation}
This does not matter how great/small are their respective associated error covariances.} In fact, if these fusing estimators are overall negatively correlated, e.g., \Newrevis{negative definite $\sum_{i<j \in \mathcal{I}} 2 w_i w_j \mathrm{MCE}_\mathbf{y}(\hat{\mathbf{x}}_i,\hat{\mathbf{x}}_j)$}, the \Newrevis{MSE reduction} will be more significant than what was stated in \eqref{eq:AAmseGain}.

\subsection{Multimodality \Newrevis{and MAP estimator}} \label{sec:InflationAnalysis}
There is a prevailing oversimplification of the relationship between the variance and the accuracy of the fused estimate since the work \citep{Mahler09erroneous}.
As indicated in Definition 1, they are not the same: a smaller variance associated with the fused estimate does not directly imply a better estimate whether in the sense of MMSE or Bayes optimality. 
In fact, the modelling of the noise with a reasonably large variance or heavy tail (which will result in a large filter estimate covariance) can accommodate model mismatches \citep{Sarkka09,VoBT13} and combat outlier \citep{Piche12,Zhu13VBstuT}.
\Newrevis{Moreover, note that 
the naive AA fusion is given by the combination of the two densities which are re-weighted forming a multimodal FMD that exhibits more than one mode but may not be merged to a unimodal density in practice. Merging should only be applied when the mixands are close significantly to each other. }
To gain insight from a simple example---which is consistent with the example used in \citep{Mahler09erroneous}--two Gaussian probability densities
with different means and covariances are considered and analyzed in Appendix \ref{Append-model-presevation}.

\Newrevis{The component merging operation may be preferable
in applications such as MMSE/EAP-based estimation {and when the fusing estimators are consistent with each other} \citep{ardeshiri2015}, but may
lead to a loss of important details of the mixture,
e.g., the mode, which is less desirable in MAP-based estimators and is not preferable when the fusing estimators are inconsistent or even conflicting with each other}.
In practice, the MAP estimator is usually implemented in a simplified way which extracts the mode of the greatest weighted fusing estimators as in \eqref{eq:appMAP} (or multiple in the case of multi-target estimation), 
rather than calculating the real mode of the fused density as a whole as in \eqref{eq:exaAAMAP}.
\begin{align}
\breve{\mathbf{x}}_\mathrm{AA}^\mathrm{MAP} & \triangleq \hat{\mathbf{x}}_{j=\mathop{\arg\max}\limits_{i \in \mathcal{I}}w_i}^\mathrm{MAP} \ist, \label{eq:appMAP}\\
\hat{\mathbf{x}}_\mathrm{AA}^\mathrm{MAP}  & \triangleq \mathop{\arg\sup}\limits_{\tilde{\mathbf{x}}\in \mathbb{R}^{d}}  f_\mathrm{AA}(\tilde{\mathbf{x}}) \ist. \label{eq:exaAAMAP}
\end{align}

The common choice \eqref{eq:appMAP} leads to a mode-preservation capacity which reinforces the tolerance of the AA fusion to inconsistent fusing estimators or even fault estimators, namely ``\textit{fault-tolerant}''. This is a unique feature of the AA fusion as compared with the other fusion approaches. 
There is a theoretical explanation. Recall that the AA and GA fusion rules symmetrically minimize the weighted sum of the directional KL divergences 
between the fusing probability distributions and the fused result as follows \citep{Kulhavy96,Abbas09,DaKai_Li_DCAI19} 
\begin{align}
  f_\text{AA}(\mathbf{x}) &= \mathop{\arg\min}\limits_{g \in \mathcal{F}_d} \sum_{i \in \mathcal{I}}{w_iD_\text{KL}\big(f_i||g\big)} \ist, \label{eq:AA_KL divergence} \\
  f_\text{GA}(\mathbf{x}) &= \mathop{\arg\min}\limits_{g: \int_{\mathbb{R}^{d}} {g(\mathbf{x}) d\mathbf{x}} =1} \sum_{i \in \mathcal{I}}{w_iD_\text{KL}\big(g||f_i\big)} \ist, \label{eq:GA_KL divergence} 
\end{align}
where {$\mathcal{F}_d$ denotes the set of scalar-valued functions in space $\mathbb{R}^{d}$: $\mathcal{F}(\mathbf{x}) = \{f: \mathbb{R}^{d} \rightarrow \mathbb{R} \}$}.

{Comparison of the two alternative forms for the KL divergence has been illustrated in \citep[p. 468-469]{Bishop06PRML} and highlighted in the viewpoint of variational inference as follows
\begin{quote}
  In practical applications, the true posterior distribution will often be multimodal, with most of the posterior mass concentrated in some number of relatively small regions of parameter space. These multiple modes may arise through nonidentifiability in the latent space or through complex nonlinear
dependence on the parameters. 
A variational treatment based on the minimization of [the forward KL divergence as on the right hand of \eqref{eq:GA_KL divergence}] will tend to find one of these modes. By contrast, if we
were to minimize [the reverse KL divergence as on the right hand of \eqref{eq:AA_KL divergence}], the resulting approximations would average across all
of the modes.
\end{quote}}

Similar illustration for the comparison of the two forms of KL divergence can also be found in \citep{ardeshiri2015}. They are members of the alpha family of divergences, manifesting zero-forcing and zero-avoiding properties, respectively \citep[p. 470]{Bishop06PRML}. \Newrevis{
The symmetry of the optimization problems \eqref{eq:AA_KL divergence} and \eqref{eq:GA_KL divergence} results in the duality of both averaging approaches \citep{Kulhavy96}.
}

\section{Information Divergence and Fusion Weights} \label{sec:information-theoretics}

\Newrevis{
A good estimator has to be accurate as much as the conservativeness can still be guaranteed. To this end, the fusing weights need to be properly designed.}
The simplest weighting solution is given by the normalized uniform weights, 
namely $\mathbf{w} = \mathbf{1}_I/I$. That is, all fusing estimators are treated equally. 
This is simple but does not distinguish the information of high quality from that of low.
For the purpose of determining the best fusion weights, the quality/performance of the fusing estimator may be measured in two aspects. The first is by the uncertainty of the estimator such as the associated covariance of the estimate pair. The second can be by its divergence relevant to the real state distribution $p(\mathbf{x})$ such as the KL divergence $D_\text{KL}\left(f\| p\right)$. This leads to two quite different ways for determining the fusing weights as to be addressed next. \Newrevis{The latter will be the focus of this section.}

Nevertheless, the fusing weights can also be determined for some other purposes, e.g., in the context of seeking consensus over a peer-to-peer network that is of the scope of this paper, 
they are often designed for ensuring fast convergence {\citep{Degroot74,Xiao04,Carvalho13}}. 
\Newrevis{Without loss of generality, we still consider the general probability density fusion and the results are expected to be extendable to the general finite mixture modeling \citep{Li23BM}}.






{\subsection{Error Covariance-based Fusion Weights}
Inspired by the naive fusion \eqref{eq:naive-x}, a heuristic weighting approach to fusing estimate pairs $(\hat{\mathbf{x}}_i,\mathbf{P}_i)$, $i \in \mathcal{I}$ can be given as
\begin{equation}\label{eq:weight-Err-Cov-P}
  w_i = \Big(\sum_{j \in \mathcal{I}}\mathrm{Tr}(\mathbf{P}_j^{-1})\Big)^{-1}\mathrm{Tr}(\mathbf{P}_i^{-1})
\end{equation}
This can be extended to the general probability density $f_i(\mathbf{x})$, $ i \in \mathcal{I}$ as follows
\begin{equation}\label{eq:weight-Err-Cov-f}
  w_i = \frac{\mathrm{Tr}\big((\int_{\mathbb{R}^{d}}{(\mathbf{x} - \mathbf{\hat{x}}_i)(\mathbf{x} - \mathbf{\hat{x}}_i)^\text{T} f_i(\mathbf{x})}d \mathbf{x})^{-1}\big)}{ \sum_{j \in \mathcal{I}} \mathrm{Tr}\big((\int_{\mathbb{R}^{d}}{(\mathbf{x} - \mathbf{\hat{x}}_j)(\mathbf{x} - \mathbf{\hat{x}}_j)^\text{T} f_j(\mathbf{x})}d \mathbf{x})^{-1}\big)}
\end{equation}
where $\hat{\mathbf{x}}_i = \int_{\mathbb{R}^{d}} \mathbf{\tilde{x}} f_i(\mathbf{\tilde{x}}) d \mathbf{\tilde{x}}$.}

\subsection{Information-theoretic Suboptimal Fusion Weights}


For a number of probability distributions ${f_i}(\mathbf{x}), i \in \mathcal{I}$, 
the KL divergence of the target distribution $p(\mathbf{x})$ relative to their average $f_\text{AA}(\mathbf{x})$  is given as
\begin{align}
 D_\text{KL}\left({f_\text{AA}}\| p\right) 
 = & \int_{\mathbb{R}^{d}} {\sum_{i \in \mathcal{I}} {w_i }{f_i}(\mathbf{x})\log \frac{{f_\text{AA}}(\mathbf{x})}{p(\mathbf{x})}\delta \mathbf{x}}   \nonumber \\
 = & \sum_{i \in \mathcal{I}} {w_i} \big( D_\text{KL}({f_i}\| p)  -  D_\text{KL}( {f_i}\| {f_\text{AA}}) \big) \label{eq:AAequation} \\
\leq & \sum_{i \in \mathcal{I}} {w_i} D_\text{KL}({f_i}\| p) \ist, \label{eq:AAbound}
\end{align}
where the equation in \eqref{eq:AAbound} holds iff $D_\text{KL}( {f_i}\| {f_\text{AA}}) = 0$, $\forall i \in \mathcal{I}$, namely, all sub-distributions ${f_i}, i \in \mathcal{I}$ are identical.


The result \eqref{eq:AAbound} has been given earlier in the textbook \citep[Theorem 4.3.2]{Blahut87}. Noticing that the KL divergence function is convex, \eqref{eq:AAbound} can also be proved by using Jensen's inequality \citep[Ch. 2.6]{Cover01informationBook}. This can be interpreted as that the average of the mixture fits the target distribution better than all mixands on average. 
We show next that optimized fusion weights will accentuate the benefit of fusion. 
Following \eqref{eq:AAequation}, the optimal solution should minimize $D_\text{KL}\left({f_\text{AA}}\| p\right)$ in order to best fit the target distribution, i.e., 
\begin{align}
  \mathbf{w}_\text{opt} 
  = &  \mathop{\arg\min}\limits_{\mathbf{w} \in \mathbb{W}} \sum_{i \in \mathcal{I}} w_i \big( D_\text{KL}({f_i}\| p)  - D_\text{KL}( {f_i}\| f_\text{AA}) \big) \ist. \label{eq:min_W}
\end{align}

As shown above, the component that fits the target distribution better (corresponding to smaller $D_\text{KL}({f_i}\| p)$ and greater $D_\text{KL}( {f_i}\| f_\text{AA})$) 
deserves a greater fusing weight. 
However, the target distribution $p(\mathbf{x})$ is typically unknown or too complicated to be practically useful, so is $D_\text{KL}({f_i}\| p)$. It is also obvious that even if the real distribution $p(\mathbf{x})$ is available, the knowledge $D_\text{KL}({f_i}\| p) < D_\text{KL}({f_j}\| p)$, $\forall j \neq i$ does not necessarily result in $w_i =1, w_j=0, \forall j \neq i$. It depends on $D_\text{KL}( {f_i}\| f_j), \forall j \neq i$. In other words, when the fusing weights are properly designed, the arithmetically averaged mixture \Newrevis{or their merge using \eqref{eq:best-fit-GM-Gaussian}} may fit the target distribution better than the best component. This can be easily illustrated by examples as given in Fig.~\ref{fig:OptAAfusion} \Newrevis{where the merging operation is performed to the AA fusion result}.
Only in few certain cases, 
the optimal result is given by the single best density.

A simplified, practically operable, alternative is ignoring the former part in \eqref{eq:min_W} which will then reduce to the following suboptimal {constrained} maximization problem
\begin{equation}
  \mathbf{w}_\text{subopt} =  \mathop{\arg\max}\limits_{\mathbf{w} \in \mathbb{W}}  \sum_{i \in \mathcal{I}} w_i D_\text{KL}( {f_i}\| f_\text{AA}) \ist. \label{eq:entropyMax1}
\end{equation}
where the weights are constrained in the weight space $\mathbb{W}$: $w_i \geq 0, \forall i \in \mathcal{I}, \sum_{i \in \mathcal{I}} w_i = 1$.

{To gain more insight about the above optimization problem, we here address the necessary and sufficient conditions for the optimal solution. First, the problem can be more formally written as
\begin{align}
\mathop{\max} \hspace{2mm} & \sum_{i \in \mathcal{I}} w_i D_\text{KL}( {f_i}\| f_\text{AA}) \\
\text{s.t.} \hspace{2mm}  & \sum_{i \in \mathcal{I}} w_i = 1 \\
 & w_i \geq 0, i=1,2,..., I
\end{align}
for which one can construct the following Lagrangian function
\begin{equation}\label{eq:Lagrangian}
  L( \{w_i\}, \lambda, \{\mu_i\}) = \sum_{i \in \mathcal{I}} w_i D_\text{KL}( {f_i}\| f_\text{AA}) + \lambda(\sum_{i \in \mathcal{I}} w_i - 1) + \sum_{i \in \mathcal{I}} \mu_i w_i
\end{equation}
where $\lambda$ and $\{\mu_i\}, i \in \mathcal{I}$ are the Lagrangian multipliers.}

{The Karush-Kuhn-Tucker conditions for the above constrained optimization problem, which are necessary but not sufficient conditions for the optimal solution, are given as follows \begin{align}
\frac{\partial L(\{w_i\}, \lambda, \{\mu_i\}) }{\partial w_i}  & =0,  i=1,2,..., I \\
\sum_{i \in \mathcal{I}} w_i - 1 & =0  \\
w_i & \geq 0, i=1,2,..., I \\
\mu_i & \geq 0, i=1,2,..., I \\
\mu_i w_i & = 0, i=1,2,..., I
\end{align}
}

\begin{lemma} \label{lemma:mid-dist}
\Newrevis{A sufficient condition for $\mathbf{w}^* = \mathbf{w}_\mathrm{subopt}$ where $\mathbf{w}_\mathrm{subopt}$ is given in \eqref{eq:entropyMax1} can be given by, $\forall i, j \in \mathcal{I}$
\begin{equation}
 D_\text{KL}\big( {f_i}\| f_\text{AA}(\mathbf{w}^*)\big) = D_\text{KL}\big( {f_j}\| f_\text{AA}(\mathbf{w}^*)\big) \ist. \label{eq:AA-MaxMin-stationaryPoint}
\end{equation}}
\end{lemma}

\Newrevis{
\begin{proof}
The proof is given in Appendix \ref{Append-KLD}.
\end{proof}
}

\begin{figure}
\centering
\centerline{\includegraphics[width=12cm]{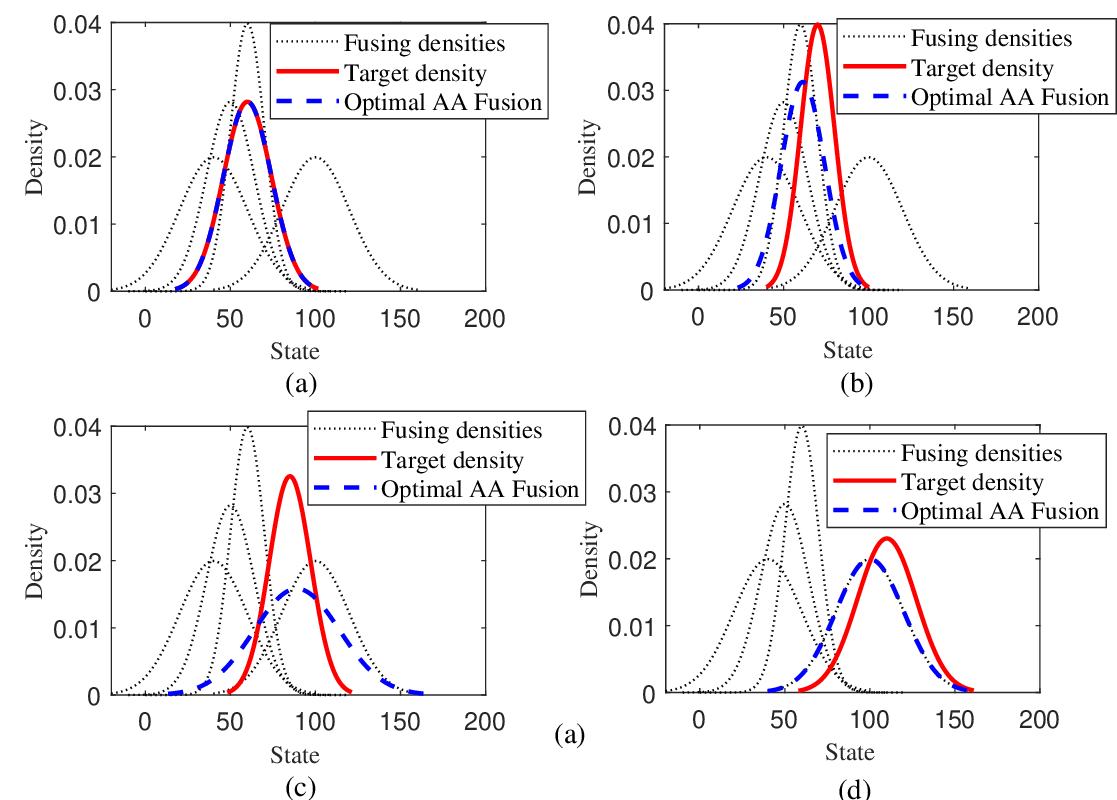}}
\caption{Optimal AA fusion of four Gaussian densities (each given by a black dotted line) to best fit the target density (red solid line) in four different cases, resulting in the optimally merged Gaussian density (blue dashed line). 
} \label{fig:OptAAfusion}
\vspace{-2mm}
\end{figure}

The suboptimal, practically operable, optimization given by \eqref{eq:entropyMax1} assigns a greater fusing weight to the distribution that diverges more from the others. This can be referred to as a \textit{diversity preference} solution. {Similarly, \citep[Section VII-A]{Koliander22} presents a solution that 
assigns a nonnegative score $\gamma_k$ to each component pdf $q_i(\theta)$ which is inversely related to the maximum discrepancy between $q_i(\theta)$ and the others.}
Alternatively, one may resort to other information measures 
to assign higher weights to the components that fit the data better {\citep{Genest90weightsLOP,DeGroot91optiWeights}}. For example, given the new observation data $\mathbf{y}$ and the likelihood function $p_i(\mathbf{y}|\mathbf{x})$ of the fusing estimator $i \in \mathcal{I}$, the fusing weights can be defined as, c.f.,  
\begin{equation}\label{eq:likelihood-weights}
   w_i = \frac{p_i(\mathbf{y}|\mathbf{x})}{\sum_{j \in \mathcal{I}}p_j(\mathbf{y}|\mathbf{x})} \ist.
\end{equation}

\subsection{\Newrevis{Max-Min Optimization for suboptimal AA and GA Fusion}}
Recall the divergence minimization \eqref{eq:AA_KL divergence} that the AA fusion admits \citep{DaKai_Li_DCAI19,Li19Bernoulli}.
Now, combining \eqref{eq:AA_KL divergence} with \eqref{eq:entropyMax1} yields joint optimization of the fusing function and fusing weights as follows
\begin{equation}
{(\mathbf{w}_\text{subopt},f_\text{AA})} = 
  \arg\mathop{\max}\limits_{{\mathbf{w} \in \mathbb{W}}} \mathop{\min}\limits_{{g \in \mathcal{F}_d}} \sum_{i \in \mathcal{I}} w_i D_\text{KL}( {f_i}\| g)  \ist. \label{eq:JiontOpt-AA}
\end{equation}

This variational fusion problem \eqref{eq:JiontOpt-AA} resembles that for geometric average (GA) fusion \citep{Nielsen13Chernoff,Uney19consistency}, i.e.,
\begin{equation}
{(\mathbf{w}_\text{subopt},f_\text{GA})} = \arg\mathop{\max}\limits_{{\mathbf{w} \in \mathbb{W}}} \mathop{\min}\limits_{{g \in \mathcal{F}_d}} \sum_{i \in \mathcal{I}} w_i D_\text{KL}(g \|{f_i}).  \label{eq:JiontOpt-GA}
\end{equation}
Results analogous to \eqref{eq:AA-MaxMin-stationaryPoint} for the GA fusion is given as
$D_\text{KL}(f_\text{GA}(\mathbf{w}_\text{subopt})\| {f_i})  = D_\text{KL}(f_\text{GA}(\mathbf{w}_\text{subopt})\| {f_j})$
which has been earlier pointed out in \Newrevis{one way \citep{Hurley02,Julier06} or another \citep{Ajgl15acc}} but so far there is no explicit proof of the sufficient condition as given in Lemma \ref{lemma:mid-dist}.
This suboptimal GA fusion is related to the Chernoff information/fusion \Newrevis{\citep{Cover01informationBook,Ahmed12,Ajgl15acc}}.
\Newrevis{However, no earlier derivation has been connected with the optimal solution \eqref{eq:min_W}, by which we notice the potential of such a diversity preference weighting choice in approaching the best fit of the target distribution. 
Notably, since weighted sum of the divergences is concave in the weights $\mathbf{w}$ and convex in the function $f$, the min and max functions can be switched without changing the final results, namely,
$$\arg\mathop{\max}\limits_{{\mathbf{w} \in \mathbb{W}}} \mathop{\min}\limits_{{g \in \mathcal{F}_d}} \sum_{i \in \mathcal{I}} w_i D_\text{KL}(g \|{f_i}) =\arg\mathop{\min}\limits_{{g \in \mathcal{F}_d}} \mathop{\max}\limits_{{\mathbf{w} \in \mathbb{W}}} \sum_{i \in \mathcal{I}} w_i D_\text{KL}(g \|{f_i})$$
 This may then lead to another interpretation of the optimization, i.e., minimizing the maximal divergence, analogous with the case of GA fusion \citep{Ajgl15acc}.}

\subsection{Case Study: Gaussian Fusion}
We now consider the special case of Gaussian PDF. 
The KL divergence of $f_1(\mathbf{x})\triangleq \mathcal{N}(\mathbf{x};\mathbf{\mu}_1,\mathbf{P}_1)$ relative to $f_2(\mathbf{x})\triangleq \mathcal{N}(\mathbf{x};\mathbf{\mu}_2,\mathbf{P}_2)$ is given as
\begin{align} 
D_\text{KL}\big(\mathcal{N}(\mathbf{\mu}_1,\mathbf{P}_1)\| \mathcal{N}(\mathbf{\mu}_2,\mathbf{P}_2)\big) 
= \frac{1}{2}\bigg[{\mathrm {tr}}\big(\mathbf{P}_{2}^{-1}\mathbf{P}_{1}\big) -d + \log{\det(\mathbf{P}_{2})\over \det(\mathbf{P}_{1})} + \|{\mathbf{\mu}}_{1}-{\mathbf{\mu}}_{2}\|^2_{\mathbf{P}_2} \bigg] \ist, \nonumber
\end{align}
where $d$ is the demission of $\mathbf{x}$ and $\|{\mathbf{\mu}}_{1}-{\mathbf{\mu}}_{2}\|^2_{\mathbf{P}} \triangleq ({\mathbf{\mu}}_{1}-{\mathbf{\mu}}_{2})^\mathrm{T}\mathbf{P}^{-1}({\mathbf{\mu}}_{1}-{\mathbf{\mu}}_{2})$.

Unfortunately, there is no such analytical expression for the KL divergence between two GMs {or even between a Gaussian and a GM.} 
Despite the Monte Carlo sampling method, a number of approximate, exactly-expressed approaches have been investigated in the literature. 
In the following we consider \Newrevis{a moment-matching approach}. 
{We note that the approximation in either case will usually be less accurate if more sensors are involved. Then, a simple, heuristic solution that extends the proposed pairwise AA fusion to the multiple sensors case is to perform the fusion to all sensors one by one in order. Another heuristic solution is to cluster all sensors into different groups so that the fusion can be performed with fewer sensors and then their fusion results are (re-grouped if necessary and) fused till the final result is obtained.}

The \Newrevis{moment-matching approach} is merging the mixture to a single Gaussian PDF, or to say, fitting the GM PDF by a single Gaussian PDF. Then, the divergence of two GMs or between a Gaussian PDF and a GM can be approximated by that between their best-fitting Gaussian PDFs. 
As given in Lemma 3, the moment fitting Gaussian PDF for a GM ${f_i}(\mathbf{x}) = \mathcal{N}(\mathbf{x};\mathbf{\mu}_i,\mathbf{P}_i), i \in \mathcal{I}$ is $f_\text{AA,merged}(\mathbf{x}) = \mathcal{N}(\mathbf{x};\mathbf{\mu}_\text{AA},\mathbf{P}_\text{AA})$, where ${\mathbf{\mu}}_\text{AA} = \sum_{i \in \mathcal{I}} w_i\mathbf{\mu}_i, \mathbf{P}_\text{AA} = \sum_{i \in \mathcal{I}} w_i \big(\mathbf{P}_i + (\mathbf{\mu}_\mathrm{AA}-\mathbf{\mu}_i)(\cdot)^\mathrm{T} \big)$.
Using $f_\text{AA,merged}(\mathbf{x})$ for approximately fitting the target Gaussian PDF ${p}(\mathbf{x}) = \mathcal{N}(\mathbf{x};\mathbf{\mu},\mathbf{P})$ yields
\begin{align}
  \mathbf{w}_\text{opt} \approx \mathop{\arg\min}\limits_{\mathbf{w} \in \mathbb{W}}  &  D_\text{KL}\big( \mathcal{N}(\mathbf{\mu}_\text{AA},\mathbf{P}_\text{AA}) \| \mathcal{N}(\mathbf{\mu},\mathbf{P}) \big)  \nonumber \\
= \mathop{\arg\min}\limits_{\mathbf{w} \in \mathbb{W}} &\bigg[{\mathrm {tr}}\big(\mathbf{P}^{-1}\mathbf{P}_\text{AA}\big)  + \log{\det(\mathbf{P})\over \det(\mathbf{P}_\text{AA})}
 +\|{\mathbf{\mu}}_\text{AA}-{\mathbf{\mu}}\|^2_{\mathbf{P}}\bigg] \ist. \label{eq:optimalWeightGassuain}
\end{align}

When 
the diversity preference solution as given in \eqref{eq:entropyMax1} is adopted, one has
\begin{align}
  \mathbf{w}_\text{subopt} \approx \mathop{\arg\max}\limits_{\mathbf{w} \in \mathbb{W}} &  \sum_{i \in \mathcal{I}} w_i D_\text{KL}\big( f_i \| {f_\mathrm{AA,merged}}\big)  \nonumber \\
= \mathop{\arg\max}\limits_{\mathbf{w} \in \mathbb{W}} &  \sum_{i \in \mathcal{I}} w_i \bigg[{\mathrm {tr}}\big(\mathbf{P}_\text{AA}^{-1}\mathbf{P}_i\big)  + \log{\det(\mathbf{P}_\text{AA})\over \det(\mathbf{P}_i)} 
 +\|{\mathbf{\mu}}_i-{\mathbf{\mu}}_\text{AA}\|^2_{\mathbf{P}_\text{AA}} \bigg] \ist. \label{eq:suboptimalWeightGassuain}
\end{align}

Note that as explained in section \ref{sec:InflationAnalysis} and in \citep{ardeshiri2015}, the merging operation may lead to a loss of the important feature of the mixture, such as the mode.
More general result of the KL divergence between multivariate generalized Gaussian distributions can be found in \citep{Bouhlel19}.

\section{Simulations}
\label{sec:simulation}
We considered two representative single-target tracking scenarios, based on either linear or nonlinear state space models. In each scenario, the simulation is performed for 100 Monte Carlo runs, each having 100 filtering steps. 
Here, we limit our simulation study to the benchmark case having no false and missing data and having merely two sensors in order to gain the insight of the performance of these fusion approaches in perfectly modelled scenarios. This does not so match the purpose of fault-tolerant fusion such as the CU approach that is designed particularly for fusion involving inconsistent estimators.

The root MSE (RMSE) is a suitable metric for accuracy evaluation. 
The base filters we adopted are the benchmark Kalman filter (KF) / cubature KF (CKF) \citep{Arasaratnam09} (for the linear and nonlinear models, respectively) and the sampling importance resampling (SIR) particle filter. It is known that for the linear Gaussian model, the KF provides the exact, optimal solution. The SIR filter is a Bayesian filter based on Monte Carlo simulation for which we use $N_p=200, 500$ particles in the linear and nonlinear model, respectively.

Comparison fusion methods are the NF, the CI fusion {using \eqref{eq:CI-x}-\eqref{eq:CI-w}}, the CU fusion using \eqref{eq:CU_P-max} and the AA fusion. {The AA fusion approach may use the suboptimal weights \eqref{eq:suboptimalWeightGassuain}(which is referred to as the default AA fusion) or the heuristic error-covariance-based weights \eqref{eq:weight-Err-Cov-P}/\eqref{eq:weight-Err-Cov-f}. They are referred to as suboptimal weights and Cov-weights, respectively, and will be compared with some ad-hoc, fixed weights given in advance}. In addition to these density fusion approaches we also consider the noncooperative single-sensor KF/CKF/SIR filters (that use the measurements of Sensor 1 only) and the iterated-corrector (IC) algorithm for fusing the measurements of two sensors. In the IC approach, the measurements of two sensors are 
used in sequence in the filters, which amounts to calculating the joint likelihood by multiplying their respective likelihoods, i.e.,
\begin{equation}\label{eq:jointLikelihood}
  p(\mathbf{y}_{1,k},\mathbf{y}_{2,k}|\mathbf{x}_k) = p(\mathbf{y}_{1,k}|\mathbf{x}_k)p(\mathbf{y}_{1,k}|\mathbf{x}_k)\ist,
\end{equation}
where $\mathbf{y}_{1,k}$ and $\mathbf{y}_{2,k}$ are the measurements of the target generated at sensor 1 and 2 at time $k$, respectively.

If the measurements of two sensors are uncorrelated, the IC approach is Bayesian optimal. 
The posterior densities yielded by the KFs/CKFs are Gaussian PDFs of which the NF, CI, CU fusion all result in a single Gaussian PDF. In contrast, the AA fusion of two Gaussian PDFs is a GM of two components for which we apply the merging scheme \citep{Salmond09} to maintain closed-form Gaussian recursion as required in KFs/CKFs. \Newrevis{We note that this will cause extra GM merging error}. For the SIR filters, the AA fusion of two particle distributions remains a particle distribution (of a larger size namely $2N_p$ if two are merged into one) which is the advantage of the AA fusion.
However, in order to apply the NF, CI and CU fusion approach to the SIR filters, it is necessary to convert those particles to a parametric distribution for which we apply the Gaussian PDF. This particles-to-Gaussian conversion is also used in the AA fusion in order to calculate the suboptimal fusion weight \eqref{eq:suboptimalWeightGassuain}.
After the fusion, $N_p$ new particles are re-sampled from the fused Gaussian PDF. 

{Matlab codes for reproducing the simulations are available at the following URL:\\ https://sites.google.com/site/tianchengli85/matlab-codes/aa-fusion.}

\subsection{Linear Scenario}
In our first simulation, the state of the target $\mathbf{x}_k=[p_{x,k},\dot{p}_{x,k},p_{y,k},\dot{p}_{y,k}]^\mathrm{T}$ consists of planar position $[p_{x,k},p_{y,k}]^\mathrm{T}$ and velocity. $[\dot{p}_{x,k},\dot{p}_{y,k}]^\mathrm{T}$.
At time $k=0$, it is randomly initialized as $\mathbf{x}_0 \sim \mathcal{N}(\mathbf{x};\mathbf{\mu}_0,\mathbf{P}_0)$,
where $\mathbf{\mu}_0 = [1000\textrm{m},20\textrm{m}/\textrm{s},1000\textrm{m},0\textrm{m}/\textrm{s}]^\mathrm{T}$ with $\mathbf{P}_0 =$ diag$\{[500\textrm{m}^2, 50\textrm{m}^2/\textrm{s}^2$, $ 500\textrm{m}^2, 50\textrm{m}^2/\textrm{s}^2]\}$, where diag$\{\mathbf{a}\}$ represents a diagonal matrix with diagonal $\mathbf{a}$.
The target moves following a nearly constant velocity motion given as (with the sampling interval $\Delta =1$s)
\begin{equation}\label{eq:Simu_TargetDynamic}
\mathbf{x}_k= \left[ \begin{array}{cccc}
1 & \Delta & 0 & 0 \\
0 & 1 & 0 & 0 \\
0 & 0 & 1 & \Delta \\
0 & 0 & 0 & 1 \\
\end{array} \right] \mathbf{x}_{k-1}+ \left[ \begin{array}{cc}
\frac{\Delta^2}{2} & 0 \\
\Delta & 0 \\
0 & \frac{\Delta^2}{2} \\
0 & \Delta \\
\end{array} \right] \mathbf{u}_{k-1} \ist,
\end{equation}
where the process noise $\mathbf{u}_k \sim \mathcal{N}(\mathbf{u};\mathbf{0}_2\textrm{m}/\textrm{s}^2,25\mathbf{I}_2\textrm{m}^2/\textrm{s}^4)$.

Both sensors have the following linear measurement model
\begin{equation}\label{eq:linear-measurement-model}
\mathbf{y}_{k}= \left[ \begin{array}{cccc}
1 & 0 & 0 & 0 \\
0 & 0 & 1 & 0 \\
\end{array} \right] \mathbf{x}_k+ \left[ \begin{array}{c}
v_{k,1} \\
v_{k,2} \\
\end{array} \right] \ist,
\end{equation}
with $v_{k,1}$ and $v_{k,2}$ as mutually independent zero-mean Gaussian noise with the same standard deviation $R$. For sensor 1, it is fixed with $R_1=20$m {while for sensor 2, we will test different magnitudes of observation noises by setting $R_2= \rho R_1$ for different ratios $\rho =1,...,10$. In all, sensor 1 has a better quality than sensor 2.}

\begin{figure}
\centering
\centerline{\includegraphics[width=12cm]{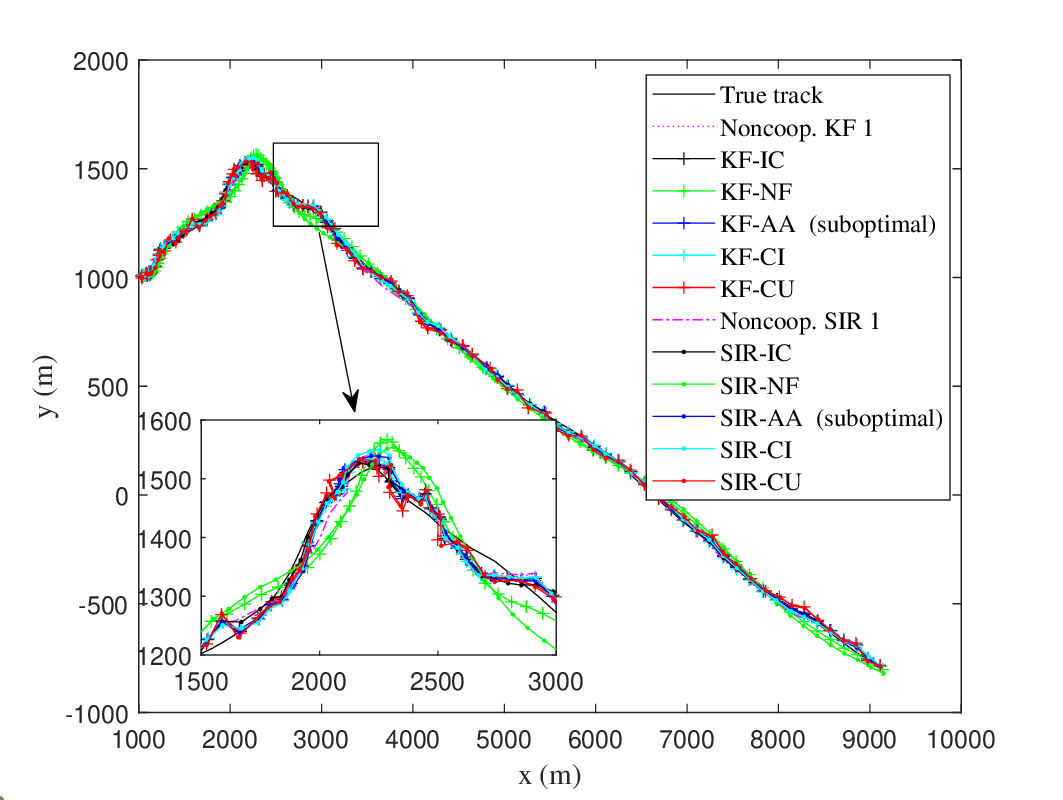}}
\caption{The real target trajectory, the estimates of the noncooperative filters and their two-sensor-fusion results when $R_2 =2R_1$ in one trial of the linear scenario.} \label{fig:Trajectory_Linear}
\vspace{-2mm}
\end{figure}

\begin{figure}
\centering
\centerline{\includegraphics[width=12 cm]{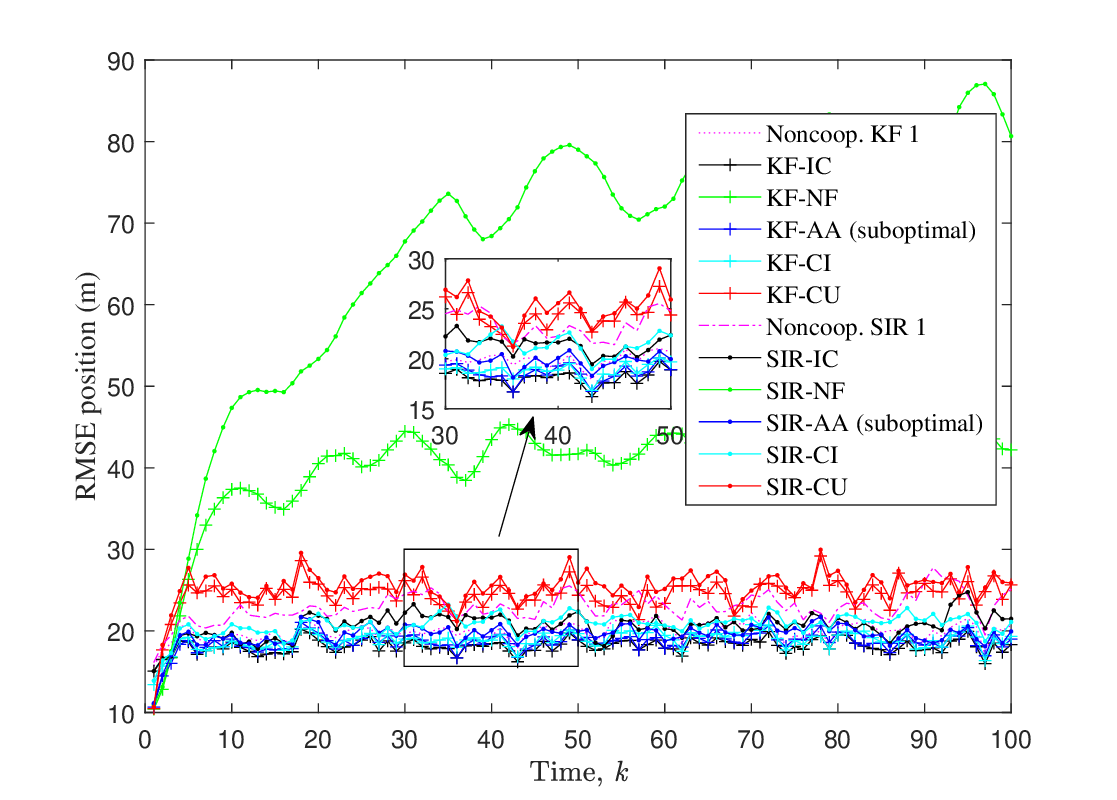}}
\caption{The position RMSEs of noncooperative and two-sensor-fusion filters when $R_2 =2R_1$ in the linear scenario.} \label{fig:posRMSE_Linear}
\vspace{-2mm}
\end{figure}

\begin{figure}
\centering
\centerline{\includegraphics[width=12 cm]{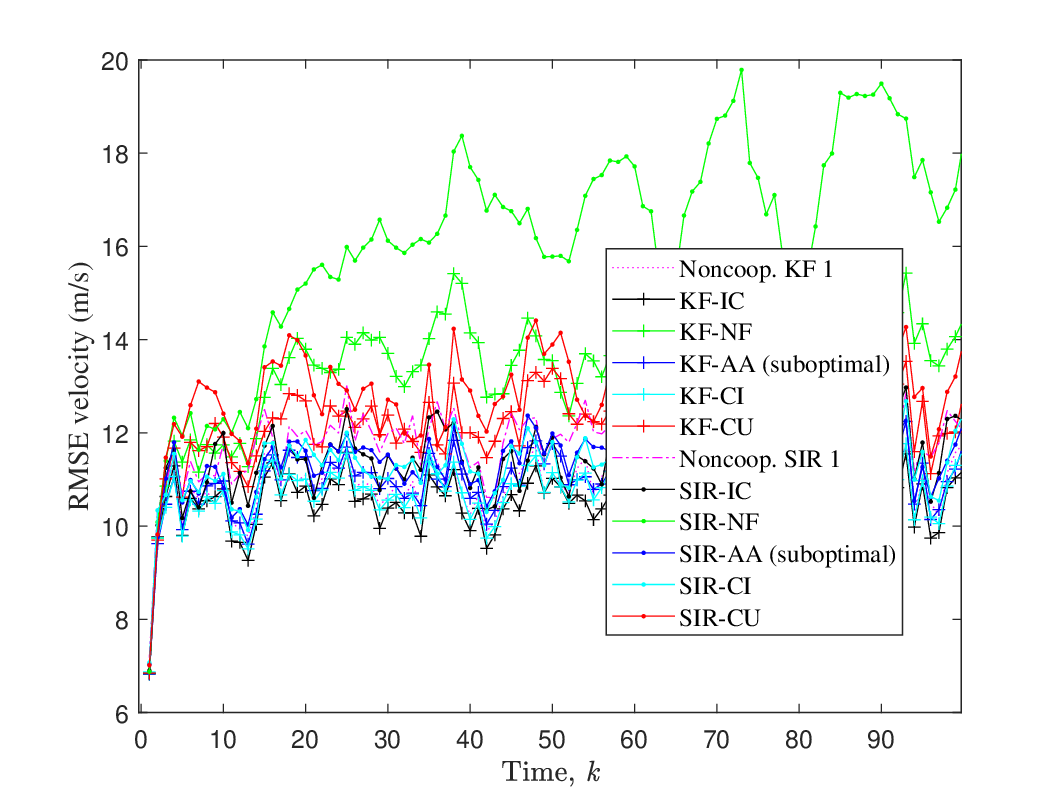}}
\caption{The velocity RMSEs of noncooperative and two-sensor-fusion filters when $R_2 =2R_1$ in the linear scenario.} \label{fig:velRMSE_Linear}
\vspace{-2mm}
\end{figure}

The real trajectory of the target and the estimates of the KFs and SIR filters in one trial when $R_2 =2R_1$ are given in Fig. \ref{fig:Trajectory_Linear}. The RMSEs of these filters are given in Fig. \ref{fig:posRMSE_Linear} and Fig. \ref{fig:velRMSE_Linear} in terms of the position and velocity estimation, respectively. The average RMSEs over all filtering times are given in Table \ref{tab:LinearRMSE}. {Here, the AA fusion weights are given by the suboptimal weighting approach using \eqref{eq:suboptimalWeightGassuain}.} As shown, the average performance of the NF based filters namely KF-NF and SIR-NF is the worst in all, even much worse than the noncooperative filters that apply no fusion. The CU-based filters are the second worst. This is probably because the former completely omits the cross-correlation between two sensors leading to inconsistent fused result, 
while the latter is over conservative leading to inaccurate results. As expected, the exactly Bayes-optimal KF-IC filter performs the best in all. The AA fusion performs similar with the CI fusion in both cases of KF and SIR filters, close to the best IC filters. 
When any of NF, AA, CI, CU or IC is applied for fusion or in the noncooperative mode, the KF slightly outperforms the SIR filter in this linear Gaussian scenario as expected.

\Newrevis{We also consider the correlated measurement noises among different sensors. That is, the noise is composed of two independent parts: $v_{k,j} = v_{k,j}^i + v_j^0, j=1, 2$ where $v_1^0, v_2^0 \sim \mathcal{N}(v; 0, 100\mathrm{m}^2)$ are two independent common components in $x$ and $y$ coordinates, respectively, of different sensors, $v_{k,j}^i \sim \mathcal(N)(v; 0, 300\mathrm{m}^2)$ for sensor 1 and $v_{k,j}^i \sim \mathcal(N)(v; 0, (400\rho-100)\mathrm{m}^2)$ for sensor 2. The average RMSEs over all filtering times are given in Table \ref{tab:LinearRMSE_correlated}, which is similar with the non-correlated measurement noise case as given in Table \ref{tab:LinearRMSE}, with slightly increased RMSEs of all methods except for the NF methods where the position RMSEs are slightly reduced in the correlated measurement noise case.}

\begin{table}[t!]
\renewcommand{\baselinestretch}{1.17}\small
\caption{Average RMSE of different filters when $R_2 =2R_1$ in the linear scenario \Newrevis{when two sensors have independent measurement noises}.}
\label{tab:LinearRMSE}
\vspace{-3mm}
\begin{center}
{\footnotesize
\begin{tabular}{|c|c|c|}
\hline
Filter & ARMSE position [m]& ARMSE velocity [m/s] \\
\hline
\hline
Noncoop. KF 1 &19.95 &10.86 \\
\hline
KF-IC	&18.13 &10.51\\
\hline
KF-NF	&39.99 &13.32\\
\hline
KF-AA (suboptimal)	&18.57 &10.91\\
\hline
KF-CI 	&18.85 &10.70\\
\hline
KF-CU 	&24.36 &12.06  \\
\hline
Noncoop. SIR 1 &22.90 &11.74 \\
\hline
SIR-IC&20.62 &11.36\\
\hline
SIR-NF	&66.07 &16.16\\
\hline
SIR-AA (suboptimal)	&19.55 &11.30 \\
\hline
SIR-CI 	&20.76&11.23 \\
\hline
SIR-CU 	&25.42 &12.72 \\
\hline
\end{tabular}
}
\end{center}
\vspace{-2.5mm}
\end{table}

\begin{table}[t!]
\Newrevis{
\renewcommand{\baselinestretch}{1.17}\small
\caption{Average RMSE of different filters when $R_2 =2R_1$ in the linear scenario when two sensors have correlated measurement noises.}
\label{tab:LinearRMSE_correlated}
\vspace{-3mm}
\begin{center}
{\footnotesize
\begin{tabular}{|c|c|c|}
\hline
Filter & ARMSE position [m]& ARMSE velocity [m/s] \\
\hline
\hline
Noncoop. KF 1 &20.07 &10.79 \\
\hline
KF-IC	&19.13 &10.60\\
\hline
KF-NF	&39.59 &13.15\\
\hline
KF-AA (suboptimal)	&19.60 &10.88\\
\hline
KF-CI 	&19.34 &10.67\\
\hline
KF-CU 	&24.41 &11.96  \\
\hline
Noncoop. SIR 1 &22.98 &11.75 \\
\hline
SIR-IC&21.63 &11.46\\
\hline
SIR-NF	&62.82 &15.54\\
\hline
SIR-AA (suboptimal)	&20.48 &11.22 \\
\hline
SIR-CI 	&21.33&11.19 \\
\hline
SIR-CU 	&26.46 &12.66 \\
\hline
\end{tabular}
}
\end{center}
\vspace{-2.5mm}
}
\end{table}

Further on, the average and variance of the fusion weight assigned to sensor 1 in the KF-AA fusion when $R_2 = 2 R_1$ are given in Fig. \ref{fig:fw_Linear}. The results show that the average fusion weight assigned to sensor 1 is about 0.575 with a small variance, which indicates that the weighting solution \eqref{eq:suboptimalWeightGassuain} indeed makes sense as it assigns a greater weight to the sensor of better quality. {Moreover, we test the performance of different weighting approaches for the AA fusion including the suboptimal weights \eqref{eq:suboptimalWeightGassuain}, the Cov-weights \eqref{eq:weight-Err-Cov-P} in comparison with the ad-hoc, fixed weights $w_1=0.5$ (and so $w_2=0.5$), $w_1=0.6$ (and so $w_2=0.4$) and $w_1=0.8$ (and so $w_2=0.2$) for different $\rho =1,...,10$. The ARMSEs of these filters over all filtering times are given in Fig. \ref{fig:wComp_linear} in terms of the position and velocity estimation, respectively. The results show that the fixed weight using $w_1=0.6$ (and so $w_2=0.4$) performs closely with the suboptimal weights \eqref{eq:suboptimalWeightGassuain}. It is non-surprising since the obtained suboptimal weight is just about $w_1= 0.575$ on average as shown in Fig. \ref{fig:fw_Linear}, which is close to the fixed weight $w_1=0.6$. Some surprisingly, both of them are inferior to the fixed weight $w_1=0.8$ when $\rho \geq 6$ for position estimation and even when $\rho \geq 3$ for velocity estimation. We note such a choice like $0.8$ is ad-hoc and by no means to be the best, or even a good choice in any other scenarios. In fact, it performed the worst in all when $\rho <3$ in terms of position estimation and $\rho \approx 1$ for velocity estimation. On the other hand, the heuristic, error-covariance-based weights \eqref{eq:weight-Err-Cov-P} perform worse than the other weights except for the fixed weight $w_1=0.5$ in most cases. In this sense, it should be abandoned. However, in the case of homogeneous sensors or approximately (namely $\rho\approx 1$), the suboptimal weights \eqref{eq:suboptimalWeightGassuain} perform the best in all, better than all other weights including the best possible fixed weights $w_i=1/I, i=1,...,I$. }

\begin{figure}
\centering
\centerline{\includegraphics[width=12 cm]{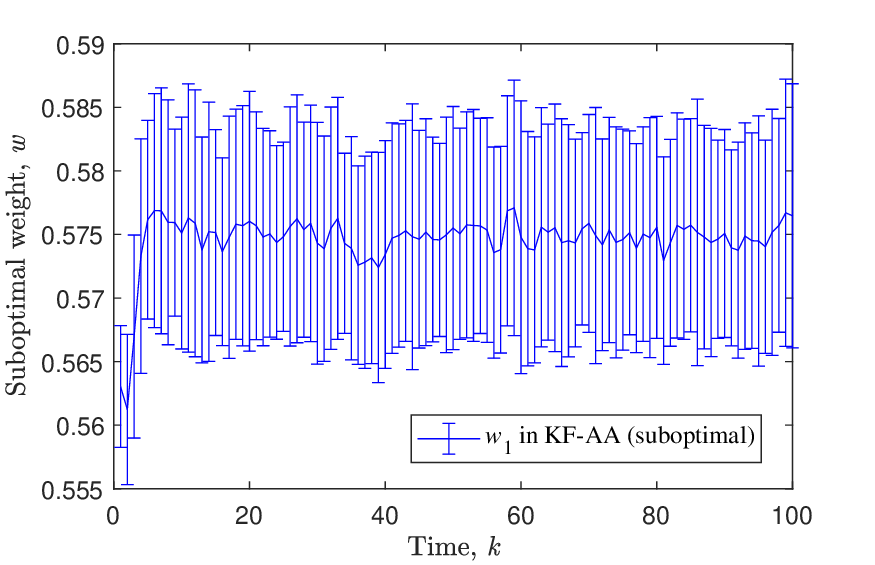}}
\caption{The average and variance of the fusion weight assigned to sensor 1 in the KF-AA fusion when $R_2 =2R_1$ in the linear scenario.} \label{fig:fw_Linear}
\vspace{-2mm}
\end{figure}

\begin{figure}
\centering
\centerline{\includegraphics[width=17 cm]{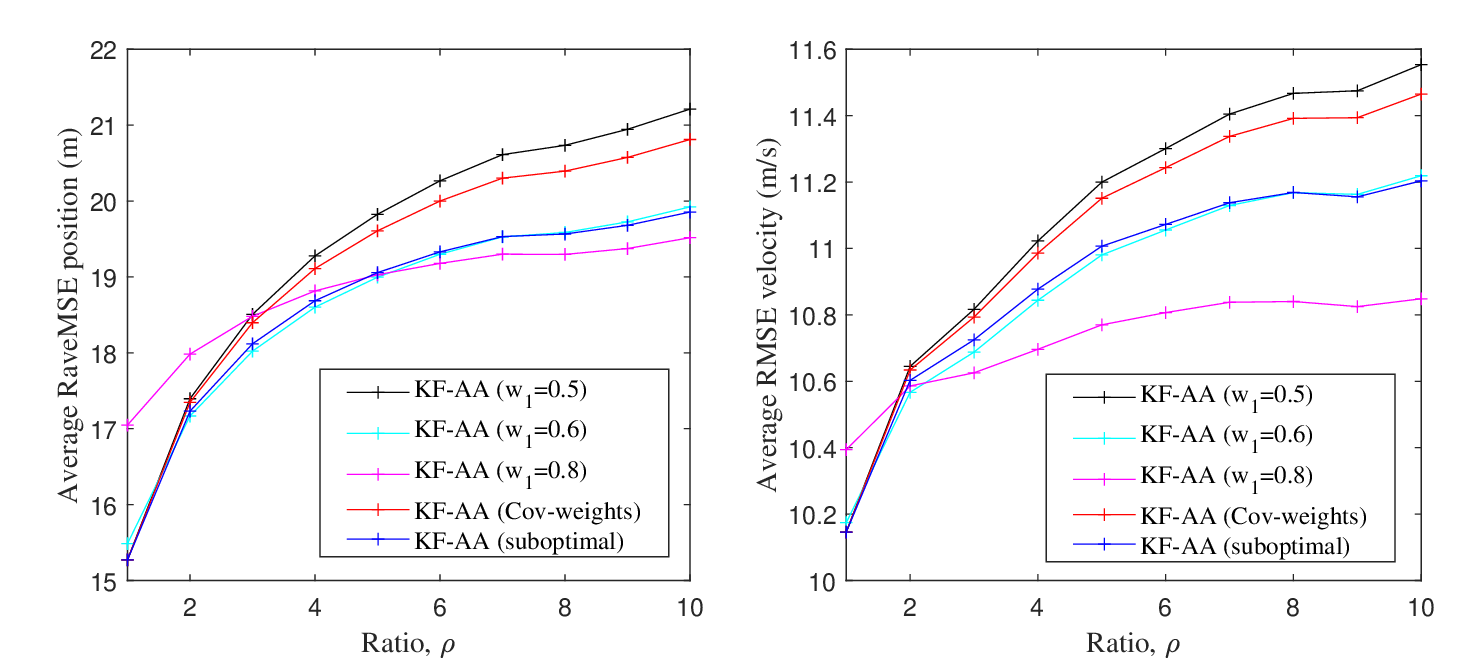}}
\caption{{The ARMSEs of the two-sensor-AA-fusion KFs using different fusing weights against different $\rho$ in the linear scenario with independent measurements among sensors.}} \label{fig:wComp_linear}
\vspace{-2mm}
\end{figure}

\subsection{Nonlinear Scenario}
In the second scenario, the target state is denoted as $\mathbf{x}_k=[p_{x,k},\dot{p}_{x,k},p_{y,k},\dot{p}_{y,k},\omega _{k}]^\text{T}$ with an additional turn rate $\omega _{k}$ as compared with that in the last simulation. It is randomly initialized as follows $\mathbf{x}_0 \sim \mathcal{N}(\mathbf{x};\mathbf{\mu}_0,\mathbf{P}_0)$,
where $\mathbf{\mu}_0 = [1000\textrm{m},20\textrm{m}/\textrm{s},1000\textrm{m},0\textrm{m}/\textrm{s},-\pi/60\textrm{rad}]^\mathrm{T}$ with $\mathbf{Q}=$ and $\mathbf{P}_0 =$ diag$\{[500\textrm{m}^2,50\textrm{m}^2/\textrm{s}^2,500\textrm{m}^2,50\textrm{m}^2/\textrm{s}^2], 0.01\textrm{rad}^2\}$.
The target moves following a coordinated turn model with a
sampling period of $\Delta =1$s and transition density $f_{k|k-1}(\mathbf{x}_{k}|\mathbf{x}_{k-1})=%
\mathcal{N}(\mathbf{x}_{k};F(\omega _{k})\mathbf{x}_{k},\mathbf{Q})${,} where%
\begin{equation}
F(\omega )=\left[
\begin{array}{ccccc}
1 & \!\!\frac{\sin \omega\Delta }{\omega } & 0 & \!\!-\frac{1-\cos  \omega\Delta }{\omega
} & 0 \\
\vspace{0.5mm}
0 & \!\!\cos  \omega\Delta & 0 & \!\!-\sin  \omega\Delta & 0 \\
\vspace{0.5mm}
0 & \!\!\frac{1-\cos  \omega\Delta}{\omega } & 1 & \!\!\frac{\sin  \omega\Delta}{\omega }
& 0 \\
\vspace{0.5mm}
0 & \!\!\sin  \omega\Delta & 0 & \!\!\cos  \omega\Delta & 0 \\
\vspace{0.5mm}
0 & 0 & 0 & 0 & 1%
\end{array}%
\right] \ist, \!\!
\end{equation}
and 
$\mathbf{Q}=\mathrm{diag}([\mathbf{I}_2\otimes\mathbf{G},\sigma _{u}^{2}])$ with
$
\mathbf{G}=%
q_1\left[
\begin{array}{cc}
\vspace{0.8mm}
\frac{\Delta^{3}}{3} & \frac{\Delta^{2}}{2} \\
\vspace{0.8mm}
\frac{\Delta^{2}}{2} & \Delta  %
\end{array}%
\right] \ist, \!\!
$
where $\otimes$ is the Kronecker product, $q_1=0.1$, and $\sigma^2_{u}=10^{-4}\text{rad}/\text{s}$.

Sensor $s \in \{1,2\}$ localized at $[ x_s; y_s]$ generates the range-bearing measurement as
\begin{equation}
\mathbf{y}_{s,k}= \begin{bmatrix}
\sqrt{( p_{x,k} \!-\! x_s )^2 + (p_{y,k} \!-\! y_s)^2} \,\,\\
\vspace{0.5mm}
\tan^{-1}\!\Big( \frac{ p_{x,k} - x_s }{ p_{y,k}  - y_s } \Big)
\end{bmatrix}
+ \begin{bmatrix} r_{s,k} \\[0.5mm] \theta_{s,k}  \end{bmatrix} \ist,
\end{equation}
where $r_{s,k}$ and $\theta_{s,k}$ are zero-mean Gaussian with standard deviation $\sigma_{s,r}$m and $\sigma_{s,\theta}$rad, respectively. They are $x_1=0$m, $y_1=0$m, $\sigma_{1,r}\!=\! 10$m and $\sigma_{1,\theta} = -\pi/180$rad for sensor 1 {while for sensor 2 positioned at $x_2=500$m, $y_2=0$m, we will test different magnitudes of observation noises by setting $[\sigma_{2,r},\sigma_{2,\theta}] = [\rho\sigma_{1,r}, \rho\sigma_{1,\theta}]$ for different ratios $\rho = 1,...,10$. }

Based on the linear and nonlinear scenarios given above, we further investigate the AA fusion approach using for sensor 2 by comparing the proposed suboptimal fusing weights with the some fixed weights specified in advance. That is, we test in the linear scenario and $\sigma_{2,r}= \rho \sigma_{1,r}, \sigma_{2,\theta}= \rho \sigma_{1,\theta}$ for different $\rho = 1,...,10$ in the nonlinear scenario, respectively.

\begin{figure}
\centering
\centerline{\includegraphics[width=12 cm]{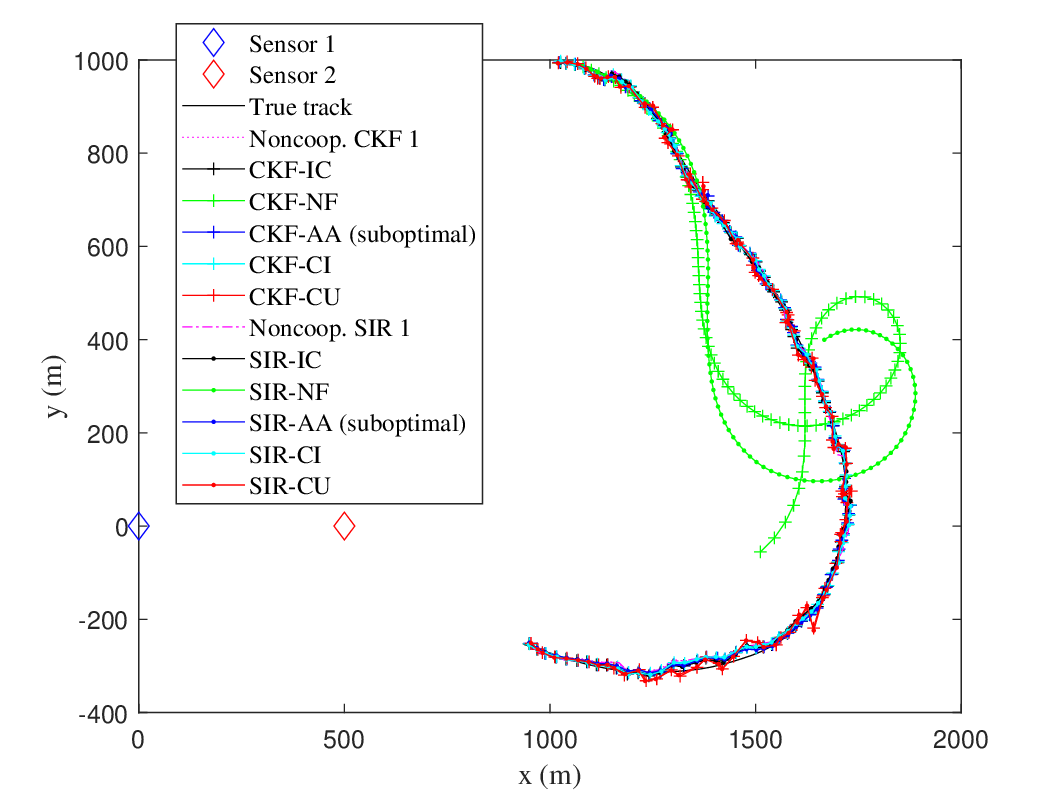}}
\caption{The position of the two sensors, the ground truth of the target trajectory, the estimates of the noncooperative filters and their two-sensor-fusion results in one trial of the nonlinear scenario.} \label{fig:Trajectory_nonLinear}
\vspace{-2mm}
\end{figure}

\begin{figure}
\centering
\centerline{\includegraphics[width=12 cm]{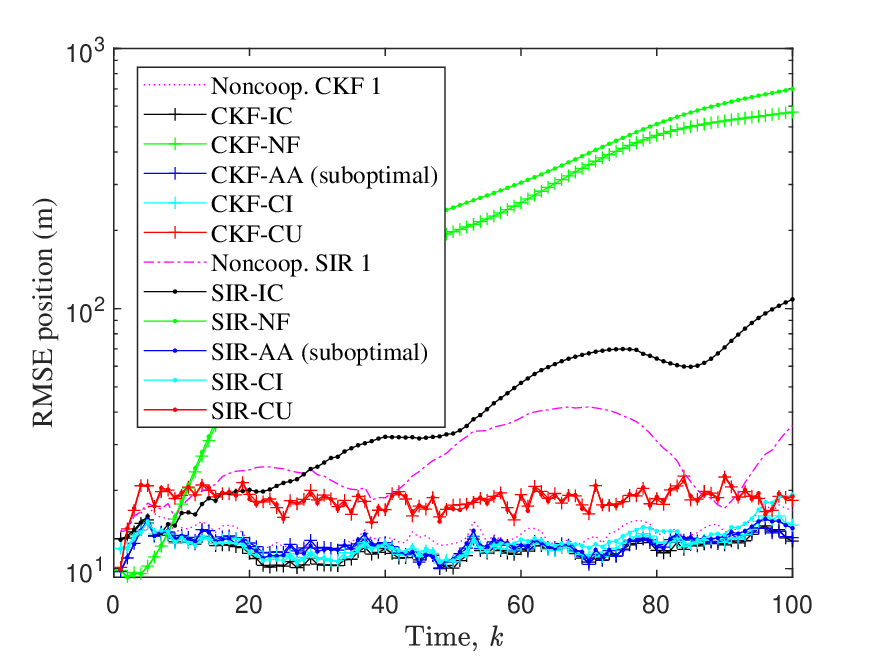}}
\caption{The position RMSEs of noncooperative and two-sensor-fusion filters in the nonlinear scenario.} \label{fig:posRMSE_nonLinear}
\vspace{-2mm}
\end{figure}

\begin{figure}
\centering
\centerline{\includegraphics[width=12 cm]{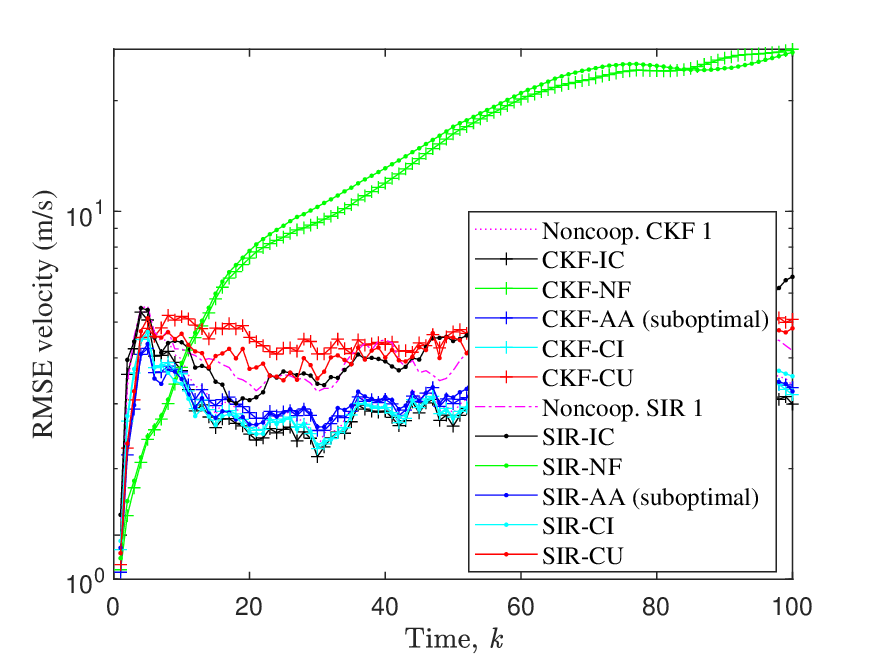}}
\caption{The velocity RMSEs of noncooperative and two-sensor-fusion filters in the nonlinear scenario.} \label{fig:velRMSE_nonLinear}
\vspace{-2mm}
\end{figure}

\begin{figure}
\centering
\centerline{\includegraphics[width=12 cm]{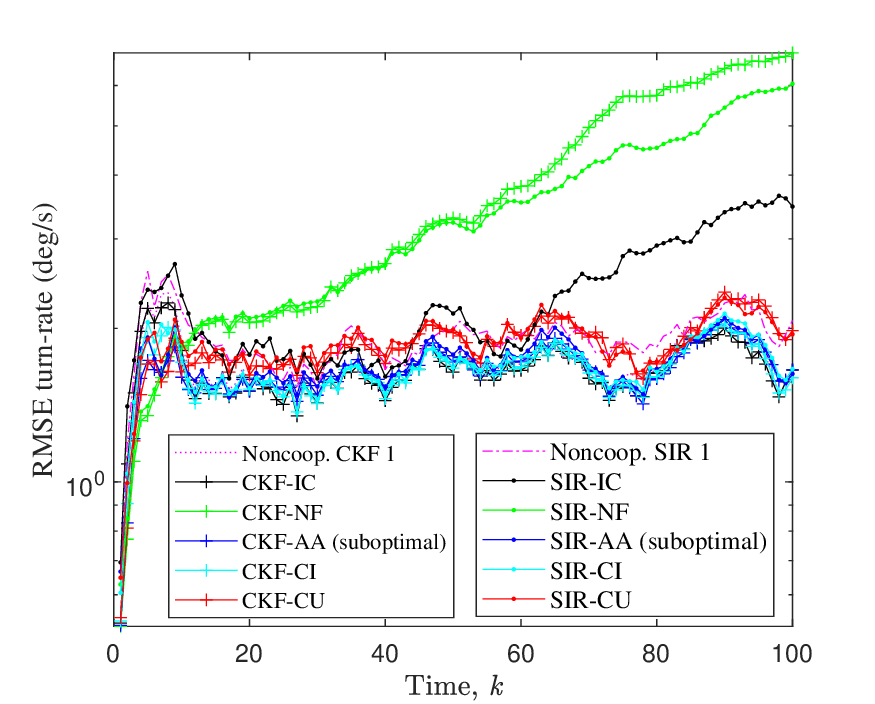}}
\caption{The turn-rate RMSEs of noncooperative and two-sensor-fusion filters in the nonlinear scenario.} \label{fig:omeRMSE_nonLinear}
\vspace{-2mm}
\end{figure}

\begin{figure}
\centering
\centerline{\includegraphics[width=12 cm]{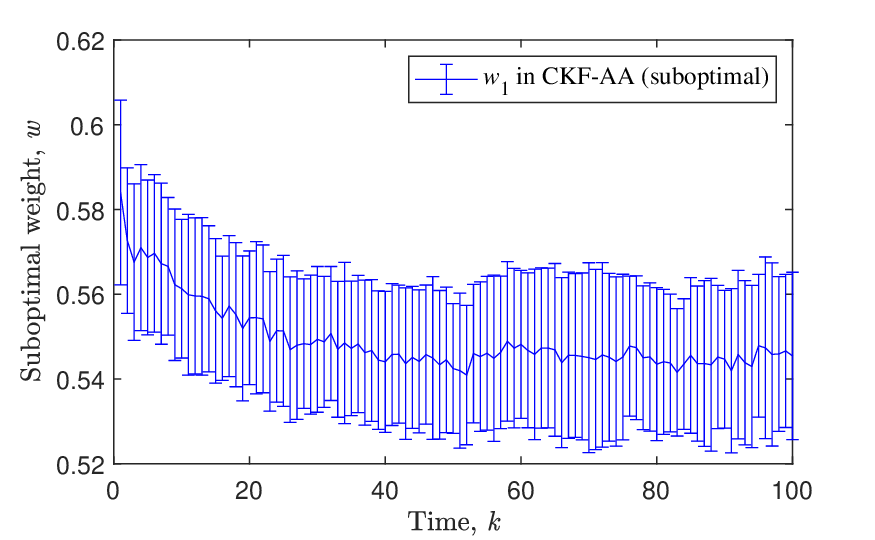}}
\caption{The average and variance of the fusion weight assigned to sensor 1 in the CKF-AA fusion in the nonlinear scenario.} \label{fig:fw_nonLinear}
\vspace{-2mm}
\end{figure}

The trajectory of the target, as well as the estimates of the noncooperative and two-sensor-fusion CKF and SIR filters, in one trial is given in Fig. \ref{fig:Trajectory_nonLinear}. The RMSEs of these noncooperative and two-sensor-fusion filters are given in Fig. \ref{fig:posRMSE_nonLinear}, Fig. \ref{fig:velRMSE_nonLinear} and Fig. \ref{fig:omeRMSE_nonLinear} in terms of the position, velocity and turn-rate estimation, respectively. The ARMSEs over all filtering times are given in Table \ref{tab:nonLinearRMSE}. These results are consistent with those in the linear scenarios. As shown in Fig. \ref{fig:Trajectory_nonLinear}, the estimates of CKF/SIR-NF filters corrupt as they diverge significantly from the real trajectory, leading to an average performance that is to a large degree worse than the others. The failure of the NF in this scenario just exposes the risk of inconsistent/over-positive fusion. 
The AA and CI fusion-based filters perform the best in whether CKF or SIR filters, close to the IC-CKF filters. Differing from what shown in the last simulation, the SIR filters perform slightly better than the CKF filters except the SIR-IC which performs much worse than the CKF-IC filter and even the SIR-CU filters. This is possibly because the joint likelihood \eqref{eq:jointLikelihood} is very informative which can easily cause particle degeneracy \citep{Li15SPM} and so deteriorate the filter. Both AA and CI fusion are free of this problem. The average and variance of the fusion weight assigned to sensor 1 in the CKF-AA fusion are given in Fig. \ref{fig:fw_nonLinear}. It shows that 
the average fusing weight assigned to sensor 1 gradually reduces from $0.584$ to $0.541$, implying that sensor 1 has been reasonably assigned with a greater fusion weight. 
The slight reduction of the fusing weight may due to that both filters perform closer with each other with the number of fusion iterations carried out between them. 


\begin{table}[t!]
 \vspace{-2.5mm}
\renewcommand{\baselinestretch}{1.17}\small
\caption{Average RMSE of different filters in the nonlinear scenario}
\label{tab:nonLinearRMSE}
\vspace{-3mm}
\begin{center}
{\footnotesize
\begin{tabular}{|c|c|c|c|}
\hline
Filter & Position [m]& Velocity [m/s] & Turn-rate [deg/s] \\
\hline
\hline
Noncoop. CKF 1	&13.80 &3.20 &1.69\\
\hline
CKF-IC	&11.94 &2.97 &1.66\\
\hline
CKF-NF	&247.49&2 15.76 &3.74\\
\hline
CKF-AA (suboptimal)	&12.42 &3.10 &1.66 \\
\hline
CKF-CI 	&12.39 &3.02 &1.95\\
\hline
CKF-CU 	&18.47 & 4.53 &1.84\\
\hline
Noncoop. SIR 1 &26.80 &4.11 &1.91\\
\hline
SIR-IC &44.40 &4.49 &2.29 \\
\hline
SIR-NF	&286.29 &16.20 &3.35 \\
\hline
SIR-AA (suboptimal)	&12.54 & 3.13 &1.71 \\
\hline
SIR-CI 	&12.81 &3.09 &1.69 \\
\hline
SIR-CU 	&18.43 &4.18 &1.88\\
\hline
\end{tabular}
}
\end{center}
\vspace{-2.5mm}
\end{table}

{Similar with the case of the linear scenario, we test the performance of different weight approaches for the AA fusion in the case of different ratios $\rho =1,...,10$. The ARMSEs of these filters over all filtering times are given in Fig. \ref{fig:wComp_nonlinear} in terms of the position and velocity estimation, respectively. 
Different from the case of the linear scenario, the performance of the AA fusion using Cov-weights overlaps with that of the fixed weights $w_1=w_2=0.5$. This indicates that the result of \eqref{eq:weight-Err-Cov-P} is close to $0.5$, i.e., the covariances obtained at two filters are very close to each other in terms of their traces. However, both of them perform the worst in all except for the fixed weight $w_1=0.8$ which performs the worst when $\rho\leq5$ for position estimation and when $\rho\leq2$ for velocity estimation. The suboptimal weights \eqref{eq:suboptimalWeightGassuain} perform similar again as the fixed weight using  $w_1=0.6$, superior to the fixed weights $w_1=0.8$ when $\rho<8$ for position estimation and when $\rho<4$ for velocity estimation. But the fixed weight $w_1=0.8$ performs better when using a larger $\rho$ (namely a significant heterogenous sensing case), i.e., Sensor 2 suffers from much more significant noises as compared with Sensor 1. In the case of homogeneous sensors, the suboptimal weights \eqref{eq:suboptimalWeightGassuain} perform the best of all, better than any fixed weights or the Cov-weights again. }

\begin{figure}
\centering
\centerline{\includegraphics[width=17 cm]{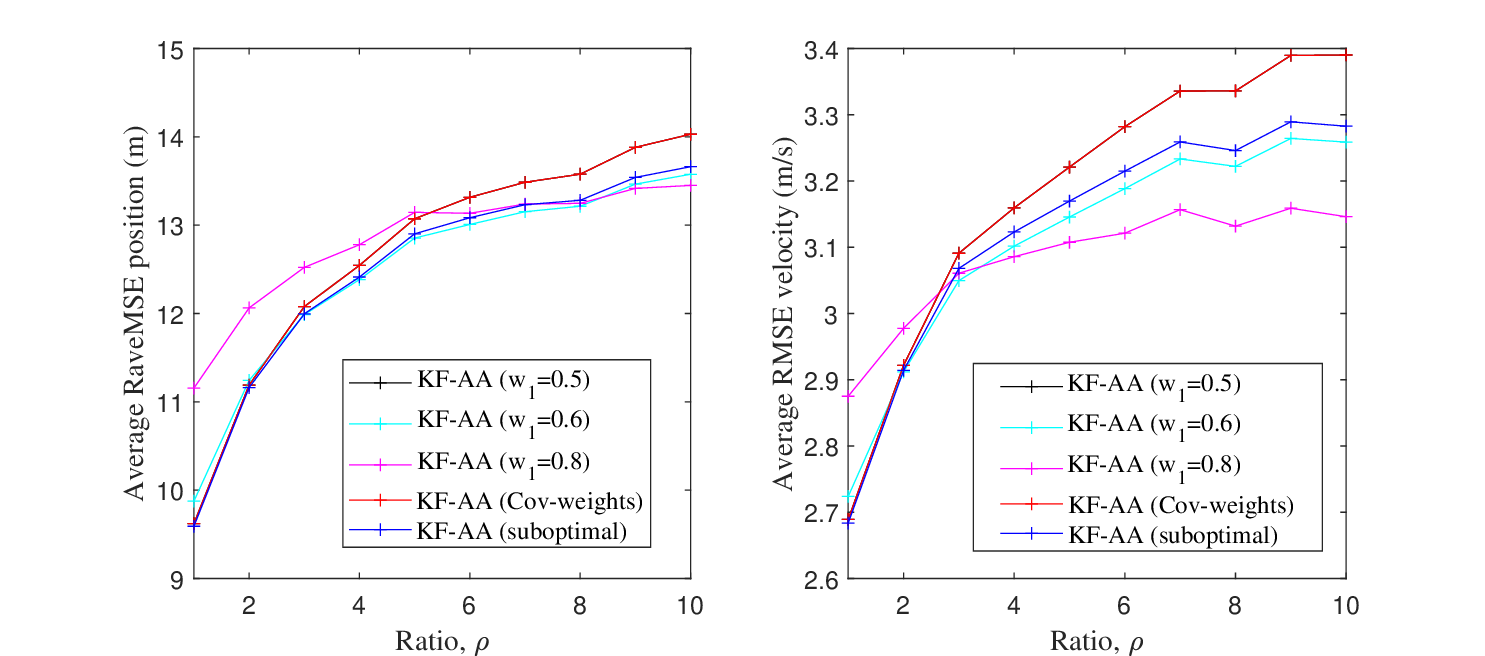}}
\caption{{The ARMSEs of the two-sensor-AA-fusion KFs using different fusing weights against different $\rho$ in the nonlinear scenario.}} \label{fig:wComp_nonlinear}
\vspace{-2mm}
\end{figure}


\section{Conclusion and Future Work} \label{sec:conclusion}
In this paper, we present some statistic and information-theoretic results on the fundamental AA density fusion, analyzing its covariance consistency, mean square error,
mode-preservation capacity, and inner relationship with some existing conservative fusion approaches including covariance union and covariance intersection. An information-theoretic, suboptimal method is proposed for online determining the fusing weights of the AA fusion approach based on the min-max optimization framework. \Newrevis{The resulted AA fusion seeks not only conservativeness but also accuracy. We may refer to this family of fusion which seeks good performance while ensuring the robustness as reliable fusion.} Representative scenarios and simulation study based on the benchmark Kalman/particle filters and GM models have been considered for verifying these theoretical findings {and for demonstrating both the effectiveness and weakness of our proposed suboptimal weighting solution.} These results may promote the conservative fusion of other advanced filters and of multi-object densities.
{However, the proposed suboptimal fusion weights ignore any a-priori knowledge about the fusing sensors and specific requirements when a large number of sensors are involved.  
For example, in the case of distributed sensor networks, the fusion weights have to take into account issues like network convergence and communication efficiency. When the sensors have significantly different qualities or different sensing frequencies, the fusion weights need to take into account them too. Moreover, it is even more challenging to properly take into account the quality of the information in fusion while maintain a high fusing convergence rate in the distributed settings. These form valuable future research. }

\Newrevis{
\section{Appendix}
\subsection[A]{Proof of Lemma \ref{lemma_AAmeanVar}} \label{Append-MSE}
}
The proof is straightforward as follows
 \begin{align}
  \mathrm{MSE}_{\hat{\mathbf{x}}_\mathrm{AA}} 
   = & \mathrm{E}_p \Big[\big(\mathbf{x}- \sum_{i \in \mathcal{I}} w_i\hat{\mathbf{x}}_i \big) \big(\cdot\big)^\mathrm{T}\Big]  \nonumber \\
= & \sum_{i \in \mathcal{I}} w_i^2 \mathrm{MSE}_{\hat{\mathbf{x}}_i} + \sum_{i<j \in \mathcal{I}} 2 w_i w_j \mathrm{MCE}_p(\hat{\mathbf{x}}_i,\hat{\mathbf{x}}_j)
 \label{eq:AA-mse-unbias} \\
 = & \sum_{i \in \mathcal{I}} w_i^2 \mathrm{MSE}_{\hat{\mathbf{x}}_i}  \label{eq:AA-mse-uncor} \\  
 \preceq & \sum_{i \in \mathcal{I}} w_i \mathrm{MSE}_{\hat{\mathbf{x}}_i} \ist, \label{eq:AA-mse-leq}   
\end{align}
where \eqref{eq:independent} was used in \eqref{eq:AA-mse-uncor}. 

\Newrevis{
To demonstrate \eqref{eq:AA-mse-unbias}-\eqref{eq:AA-mse-leq}, we consider the following statistical testing. In each trial, the real state $x$ is drawn from a 1-dimensional Gaussian distribution with PDF $p(x)=\mathcal{N}(x; 0, 4)$. Assume a number of estimators that are noisy observers of the real state, i.e., $\hat{x}_i = x + v_i, i \in \mathcal{I}$ where $v_i$ is a Gaussian noise $v_i\sim\mathcal{N}(x; 0, 1)$. 
The noises for these estimators are correlated with each other in such a way that they share a common part. That is, the noise $v_i$ is composed of two independent parts $v_i = (1-c)u_i + cu_0$, where $0\leq c \leq 1$ and $u_0, u_i\sim\mathcal{N}(x; 0, 1), i\in \mathcal{I}$ are independent of each other, namely $\mathrm{E}(u_iu_j)=0, \forall i\neq j$ but $u_0$ is the same for all estimators in each trial. 
Obviously, a larger $c \in [0, 1]$ implies a higher degree of correlation between the estimators, $c=0$ implying non-correlation and $c=1$ implying the identical estimate. Two types of fusing weights are considered: normalized uniform and random weights. In the former $w_i$ equals $1/|\mathcal{I}|$ while in the latter it is uniformly distributed in $[0, 1)$ and normalized.
The results of 10000 monte carlo trials are given in Fig. \ref{fig:MSE} for different sizes of $|\mathcal{I}|$. It shows that \eqref{eq:AA-mse-uncor} holds for $c=0$ and the equation in \eqref{eq:AA-mse-leq} holds for $c=1$ in both cases of fusing weights, regardless of the minor Monte Carlo error.
}

\begin{figure*}
\centering
\includegraphics[width=14 cm]{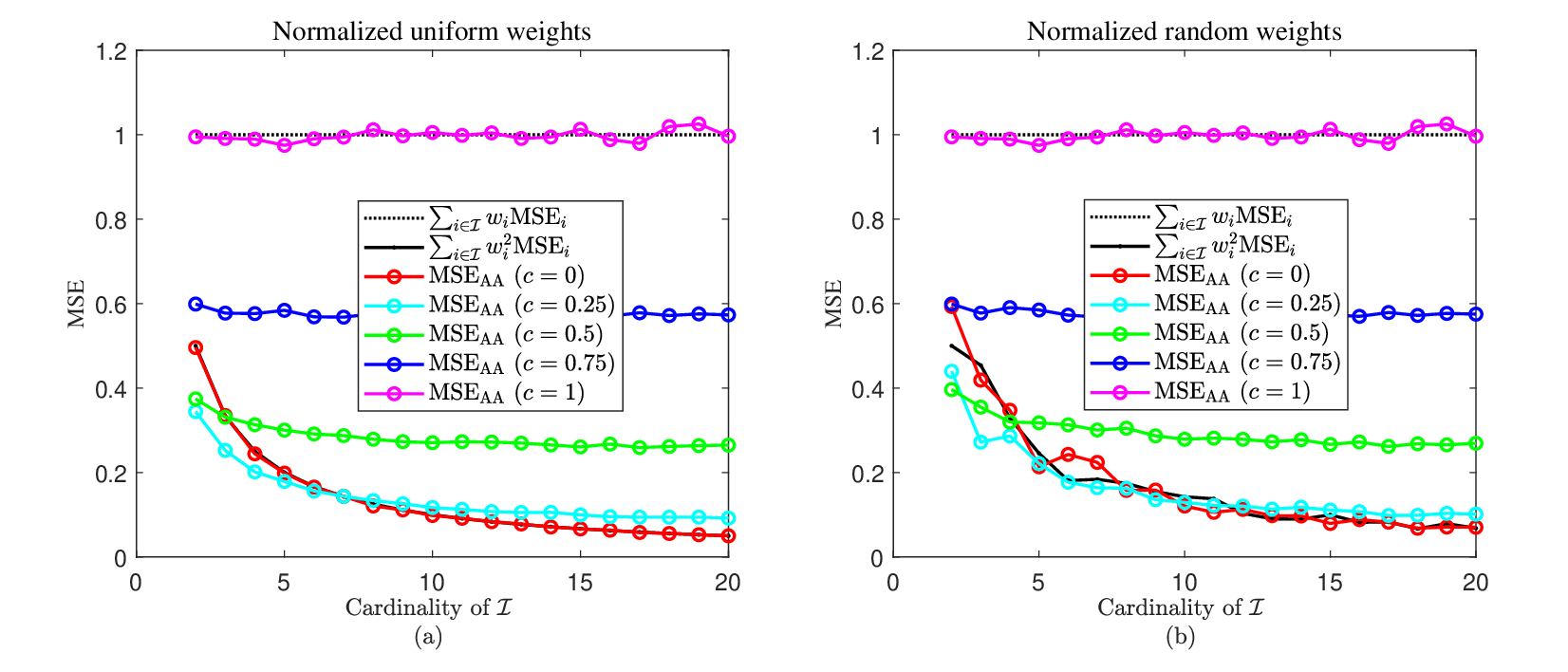}
\caption{\Newrevis{MSE of the AA of multiple estimators that are uncorrelated or correlated in different degrees. 
(a): The AA fusion using normalized uniform weights.
(b): The AA fusion using normalized random weights.
}}
\label{fig:MSE}
\vspace{-2mm}
\end{figure*}

\Newrevis{
\subsection[A]{Conservation Fusion Method Comparison: Case Study} \label{Append-model-presevation}
}
Four representative scenarios are illustrated in (a), (b), (c) and (d) of Fig.~\ref{fig:naiveCIpkAA}, respectively, where
the two Gaussian densities to be fused are visualized by two ellipses P1 and P2 that characterize the respective covariances, 
with the center points indicating the respective means. 
We consider the aforementioned fusion schemes: the NF/naive, 
un-weighted GA fusion (i.e., \eqref{eq:CI-x} and \eqref{eq:CI-P} using $w_1=w_2=0.5$), CI fusion (i.e., \eqref{eq:CI-x} and \eqref{eq:CI-P} using optimized weight as \eqref{eq:CI-w} which results in $w_1=0.3764, w_2=0.6236$ in this case), un-weighted AA fusion (i.e., \eqref{eq:AA-v-fusion} and \eqref{eq:AA-f-P} using $w_1=w_2=0.5$) with and without component merging, and two bounds of the CU fusion with fused covariance given as in \eqref{eq:CU_P-max} (referred to as CU max) and in \eqref{eq:CU_P-min} (referred to as CU min), respectively.
\Newrevis{What has been shown in blue is the merged result, which is only reasonable when the mixands are close, like in (a), but not in (d).
These results {illustrate} \eqref{eq:AAvsCU} and \eqref{eq:conservationfusionchain}.}

In all four scenarios, without the knowledge of the true target position {or the true target state probability distribution}, 
we cannot tell whether any of the seven fusion schemes is better than the others, no matter their estimate is given by the EAP or MAP.
Even in scenario (a) there is no guarantee that the target 
is localized in the intersection of P1 and P2; if it is not, then both the NF and GA/CI fusion will likely produce incorrect results. 
In scenarios (c) and (d), at least one of P1 and P2 is inconsistent {in the sense that there is a large offset from the target position
wherever the target is}. This is simply because {the statistical unbiasedness of an estimator does not necessarily mean zero or even small offset of its estimate from the truth at any particular time when the fusion is performed, i.e.,} $\mathrm{E} [ \hat{\mathbf{x}}_i] = \theta \nRightarrow \hat{\mathbf{x}}_i  = \theta$ for any $i \in \mathcal{I}$.
In these cases, the AA fusion may not merge two densities to one but keep a multimodal FMD to avoid producing inconsistent or incorrect results; in the time-series filtering applications, new data will help identify the false/inconsistent components and prune them. 

\begin{figure*}
\centering
\includegraphics[width=17 cm]{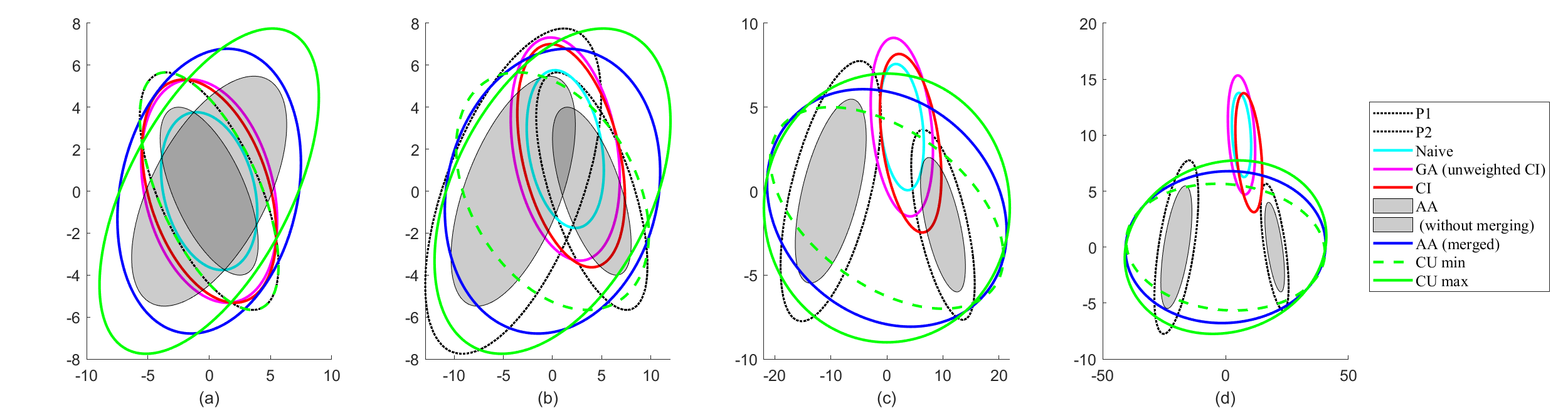}
\caption{Fusing two Gaussian densities having four different levels of divergences by different fusion methods. 
(a): two densities overlap largely and both estimators are likely to be conservative.
(b): two densities are offset from each other but still overlap somehow.
(c) and (d): two densities are greatly offset from each other and at most one estimator is conservative.
}
\label{fig:naiveCIpkAA}
\vspace{-2mm}
\end{figure*}

\Newrevis{
\subsection[A]{Proof of Lemma \ref{lemma:mid-dist}} \label{Append-KLD}
Assuming that the weight constraint $\mathbf{w}^\mathrm{T}\mathbf{1}_I = 1$ holds, we define
\begin{align}
  D_\text{KL}^{-j} & \triangleq \frac{\sum_{i \neq j} w_i  D_\text{KL}( {f_i}\| f_\text{AA})}{1-w_j} \ist, \label{eq:def_KLDpartial}
\end{align}
It is obvious that given fusing weight as satisfying \eqref{eq:AA-MaxMin-stationaryPoint}, we have
\begin{align}
  D_\text{KL}( {f_j}\| f_\text{AA}) &= D_\text{KL}^{-j} \ist, \label{eq:KLDeq}
\end{align}
The condition \eqref{eq:AA-MaxMin-stationaryPoint} also indicates that $ 0<w_j<1$ and so we can further define
\begin{align}
  {f_{-j}}(\mathbf{x}) & \triangleq \frac{{f_\text{AA}}(\mathbf{x})-w_j{f_{j}}(\mathbf{x})}{1-w_j} \ist, \label{eq:def_fpartial}
\end{align}
}and obtain
\begin{align}
  \int_{\mathbb{R}^{d}} {f_{-j}}(\mathbf{x}) d\mathbf{x}  &  = 1 \ist. \label{eq:int-f-j=0}
\end{align}


By writing $f_\text{AA}(\mathbf{x}) = w_j{f_{j}}(\mathbf{x}) + (1-w_j){f_{-j}}(\mathbf{x})$, we have
\begin{equation}
\frac{ \partial \log  {f_\text{AA}}(\mathbf{x}) }{\partial w_j}  = \frac{ {f_j}(\mathbf{x})-{f_{-j}}(\mathbf{x}) }{f_\text{AA}(\mathbf{x}) } \label{eq:Leibniz-rule}
\end{equation}

Moreover, based on \eqref{eq:AA-MaxMin-stationaryPoint}, it is easy to derive the partial derivative of the optimization function given in \eqref{eq:entropyMax1} 
\begin{align}
 \frac{ \partial \sum_{i \in \mathcal{I}} w_i  D_\text{KL}( {f_i}\| f_\text{AA}) }{\partial w_j} 
=& D_\text{KL}( {f_j}\| f_\text{AA}) - D_\text{KL}^{-j}
 + \sum_{i \in \mathcal{I}} w_i \frac{ \partial  D_\text{KL}( {f_i}\| f_\text{AA}) }{\partial w_j}  \\
=&  - \sum_{i \in \mathcal{I}} w_i  \frac{ \partial \int_{\mathbb{R}^{d}} { {f_i}(\mathbf{x})\log  {f_\text{AA}}(\mathbf{x}) \delta \mathbf{x}}  }{\partial w_j} \label{eq:KLDpartial} \\
=&  - \sum_{i \in \mathcal{I}} w_i  \int_{\mathbb{R}^{d}} { \frac{{f_i}(\mathbf{x})\big({f_j}(\mathbf{x})-{f_{-j}}(\mathbf{x})\big)}{f_\text{AA}(\mathbf{x}) }\delta \mathbf{x}} \label{eq:Leibniz-rule-callback}  \\
=& - \int_{\mathbb{R}^{d}} \big({f_j}(\mathbf{x})-{f_{-j}}(\mathbf{x})\big)\delta \mathbf{x} \nonumber \\
=& 0 \label{eq:deri=0}  \ist
\end{align}
where \eqref{eq:KLDeq} was used in \eqref{eq:KLDpartial}, \eqref{eq:Leibniz-rule} and the Leibniz integral rule were used in \eqref{eq:Leibniz-rule-callback}, and \eqref{eq:int-f-j=0} was used in \eqref{eq:deri=0}.

Further on, $\frac{ \partial^2 \sum_{i \in \mathcal{I}} w_i  D_\text{KL}( {f_i}\| f_\text{AA}) }{\partial w_j^2}  = - \frac{\sum_{i \neq j} w_i  D_\text{KL}( {f_i}\| f_\text{AA})}{(1-w_j)^2} \leq 0 $ where the last equation holds if and only if $f_i (\mathbf{x})= f_j(\mathbf{x}), \forall i \neq j$.
So, this lemma is proven.

\section*{Acknowledgment}
This work was partially supported 
 by National Natural Science Foundation of China (Grant No. 62071389), Natural Science Basic Research Program of Shaanxi Province (Program No. 2023JC-XJ-22), 
 Key Laboratory Foundation of National Defence Technology (No. JKWATR-210504) and the Fundamental Research Funds for the Central Universities.


\bibliographystyle{model1-num-names} 
\bibliography{SomeResults}

\end{document}